\documentclass{amsart}
\usepackage{amsmath, amssymb, amscd, mathtools, graphicx, amsfonts, enumerate, mathrsfs, xcolor}

\newtheorem{theorem}{Theorem}[section]
\newtheorem{proposition}[theorem]{Proposition}
\newtheorem{lemma}[theorem]{Lemma}
\newtheorem{corollary}[theorem]{Corollary}

\theoremstyle{definition}
\newtheorem{definition}[theorem]{Definition}
\newtheorem{example}[theorem]{Example}
\newtheorem{remark}[theorem]{Remark}

\numberwithin{equation}{section}

\newcommand{\N}{\mathbb{N}}                        % natural numbers
\newcommand{\Z}{\mathbb{Z}}                        % integers
\newcommand{\K}{\mathbb{K}}                        % field K
\newcommand{\R}{\mathbb{R}}                        % real numbers
\newcommand{\C}{\mathbb{C}}                        % complex numbers
\newcommand{\vp}{\varphi}                          % mapping phi
\newcommand{\OO}{\mathcal{O}}                      % structure sheaf
\newcommand{\supp}{\mathrm{supp}}                  % support of a q-tuple of power series
\newcommand{\inexp}{\mathrm{exp}}                  % initial exponent of a q-tuple of power series
\newcommand{\NN}{\mathfrak{N}}                     % diagram of initial exponents
\newcommand{\LL}{\Lambda}
\newcommand{\mm}{\mathfrak{m}}                     % maximal ideal
\newcommand{\nn}{\mathfrak{n}}                     % ideal n
\newcommand{\tx}{\tilde{x}}

\begin{document}

\title[Equisingular algebraic approximation of analytic germs]{Equisingular algebraic approximation of real and complex analytic germs}

\author{Janusz Adamus}
\address{Department of Mathematics, The University of Western Ontario, London, Ontario, Canada N6A 5B7}
\email{jadamus@uwo.ca}
\author{Aftab Patel}
\address{Department of Mathematics, The University of Western Ontario, London, Ontario, Canada N6A 5B7}
\email{apate378@uwo.ca}
\thanks{J. Adamus's research was partially supported by the Natural Sciences and Engineering Research Council of Canada}


\begin{abstract}
We show that a Cohen-Macaulay analytic singularity can be arbitrarily closely approximated by germs of Nash sets which are also Cohen-Macaulay and share the same Hilbert-Samuel function. We also prove that every analytic singularity is topologically equivalent to a Nash singularity with the same Hilbert-Samuel function.
\end{abstract}
\maketitle

%%%%%%%%%%%%%%%%%%%%%%%%%%%%%%%%%%%%%%%%%%%%%%%%%%%
%%%%%%%%%%%%%%%%%%%%% Section %%%%%%%%%%%%%%%%%%%%%
%%%%%%%%%%%%%%%%%%%%%%%%%%%%%%%%%%%%%%%%%%%%%%%%%%%
\section{Introduction}
\label{sec:intro}

One of the central problems in analytic geometry concerns algebraic approximation of analytic objects. In the present paper, we consider approximation of a (real or complex) analytic set germ $(X,a)$ by germs of Nash, or even algebraic, sets which are equisingular with $(X,a)$ in the sense of the Hilbert-Samuel function. Recall that, for an analytic germ $(X,a)$, the \emph{Hilbert-Samuel function} $H_{X,a}$ is defined as
\[
H_{X,a}(\eta)=\dim_\K\frac{\OO_{X,a}}{\mm^{\eta+1}}\,,\qquad\mathrm{for\ all\ }\eta\in\N\,,
\]
where $\OO_{X,a}$ is the local ring of $X$ at $a$, with the maximal ideal $\mm$. The Hilbert-Samuel function encodes many important algebro-geometric properties of the germ and may be regarded as a measure of its singularity. It plays a central role in resolution of singularities (see \cite{BM2}).

In the first part of this paper, we deal with singularities of special types. Namely, those whose local ring is Cohen-Macaulay, or even better, a complete intersection. We prove that a complete intersection singularity can be arbitrarily closely approximated by germs of algebraic sets which are also complete intersections and share the same Hilbert-Samuel function (Theorem~\ref{thm:CI-approx}). Polynomial approximation is not possible, in general, for Cohen-Macaulay singularities (see Example~\ref{ex:no-CM-fin-det}). The next best thing is approximation by Nash sets, that is, by sets consisting of analytic branches of algebraic sets. In Theorem~\ref{thm:CM-approx}, we show that a Cohen-Macaulay singularity can be arbitrarily closely approximated by germs of Nash sets which are also Cohen-Macaulay and share the same Hilbert-Samuel function.
These results may be combined with the main theorem of the second part, which works for arbitrary singularities.

The second part of the paper is concerned with approximation of an analytic singularity $(X,0)$ by Nash germs which are homeomorphic with $(X,0)$. We give a variant of Mostowski's theorem \cite{Most}, Theorem~\ref{thm:main}, showing that every analytic germ $(X,0)\subset(\K^n,0)$ can be arbitrarily closely approximated by a Nash germ $(\hat{X},0)\subset(\K^n,0)$, such that the pairs $(\K^n,X)$ and $(\K^n,\hat{X})$ are topologically equivalent near zero, and the Hilbert-Samuel functions $H_{X,0}$ and $H_{\hat{X},0}$ coincide.
\medskip

Let $\K=\R$ or $\C$, let $x=(x_1,\dots,x_n)$ and let $\K\{x\}$ denote the ring of convergent power series in variables $x$. If $(X,0)$ is an analytic germ at $0$ in $\K^n$, its local ring $\OO_{X,a}$ is of the form $\K\{x\}/I$ for some ideal $I$ in $\K\{x\}$.
One says that the germ $(X,0)$ is \emph{Nash} if the ideal $I$ can be generated by algebraic power series, that is, by series algebraic over the ring of polynomials $\K[x]$.

When dealing with questions about the Hilbert-Samuel function of the ring $\K\{x\}/I$, it is convenient to work with a so-called diagram of initial exponents of $I$, a combinatorial representation of the ideal $I$, denoted $\NN(I)$, which we recall in Section~\ref{sec:diagram}.
Indeed, the Hilbert-Samuel function of $\K\{x\}/I$ may be read off from the sub-level sets of (the complement of) $\NN(I)$ (Lemma~\ref{lem:HS-diagram-complement}). The diagram itself is, in turn, uniquely determined by a standard basis of $I$, which is a special generating set of $I$ (see Section~\ref{sec:s-series}). Our key tool in establishing Hilbert-Samuel equisingularity of a given germ and its approximants is a theorem of Becker \cite{Be1}, which gives a criterion for a collection $\{F_1,\dots,F_t\}\subset I$ to form a standard basis of $I$ in terms of finitely many equations that depend polynomially on the $F_i$. It is therefore well suited for an application of the classical Algebraic Artin Approximation.

We call $(X,0)$ a Cohen-Macaulay (resp. complete intersection) singularity when the local ring $\OO_{X,0}$ is Cohen-Macaulay (resp. a complete intersection); see Section~\ref{sec:CI} for definitions.
The finite determinacy of the Hilbert-Samuel function of a complete intersection follows already from the work of Srinivas and Trivedi \cite{ST}. We give a new proof of this fact here, because it can also be applied in the Cohen-Macaulay case, which is new. Roughly speaking, we combine the equivalence of Cohen-Macaulayness and flatness (Remark~\ref{rem:CM-flat}) with a corollary to Hironaka's flatness criterion (Proposition~\ref{prop:flatness-diagram}), to show that with respect to certain total ordering on $\N^n$ the diagrams of $I$ and its suitable approximation $I_\mu$ coincide. We then show (Proposition~\ref{prop:HS-stability}) that this equality implies equality of the diagrams with respect to the standard ordering, and hence equality of the Hilbert-Samuel functions.

Finally, in the proof of Theorem~\ref{thm:main}, we combine the above Becker criterion with the original strategy of Mostowski, based on P{\l}oski's parametrized Artin approximation \cite{Plo} and a theorem of Varchenko stating that the algebraic equisingularity of Zariski implies the topological equisingularity \cite{Var1}. We use the modern exposition of Mostowski's theorem, due to Bilski, Parusi{\'n}ski and Rond \cite{BPR}, where the original P{\l}oski theorem is replaced with a more powerful Theorem~\ref{thm:BPR}.
\medskip

The paper is structured as follows: In the next section, we recall Hironaka's division algorithm and its consequences for the diagrams of initial exponents. We also prove Proposition~\ref{prop:flatness-diagram}, used in the treatment of Cohen-Macaulay singularities.
In Sections~\ref{sec:s-series} and~\ref{sec:APP}, we recall our two main tools: the s-series criterion and a nested parametrized algebraic approximation theorem.
Sections~\ref{sec:reduction} and~\ref{sec:approx} contain various auxiliary results needed in the proofs of Theorems~\ref{thm:CI-approx} and~\ref{thm:CM-approx}. The latter are proved in Sections~\ref{sec:CI} and~\ref{sec:CM} respectively. The remainder of the paper is devoted to the proof of Theorem~\ref{thm:main}.
\medskip

%%%%%%%%%%%%%%%%%%%%%%%%%%%%%%%%%%%%%%%%%%%%%%%%%%%
%%%%%%%%%%%%%%%%%%%%% Section %%%%%%%%%%%%%%%%%%%%%
%%%%%%%%%%%%%%%%%%%%%%%%%%%%%%%%%%%%%%%%%%%%%%%%%%%
\section{Hironaka's division algorithm, diagram of initial exponents, and flatness}
\label{sec:diagram}

Let $\K=\R$ or $\C$. Let $A$ denote the field $\K$ or the ring $\K\{y\}$ of convergent power series in variables $y=(y_1,\dots,y_m)$.
Let $A\{z\}$ denote the ring of convergent power series in variables $z=(z_1,\dots,z_k)$ with coefficients in $A$ (i.e., $\K\{z\}$ or $\K\{y,z\}$, depending on $A$).
We will write $z^\alpha$ for $z_1^{\alpha_1}\dots z_k^{\alpha_k}$, where $\alpha=(\alpha_1,\dots,\alpha_k)\in\N^k$.

The mapping $A\ni F(y)\mapsto F(0)\in\K=A\otimes_\K\!\K$ of evaluation of the $y$ variables at $0$ induces an evaluation mapping
\[
A\{z\}\ni F=\sum_{\alpha\in\N^k}F_\alpha(y)z^\alpha\,\mapsto\,F(0)=\sum_{\alpha\in\N^k}F_\alpha(0)z^\alpha\in\K\{z\}\,.
\]
(In case $A=\K$, this is just the identity mapping.)
For an ideal $I$ in $A\{z\}$, define $I(0)\coloneqq\{F(0):F\in I\}$ in $\K\{z\}$, the \emph{evaluated ideal}.

A \emph{positive linear form} on $\K^k$ is given as $\LL(\alpha)=\sum_{j=1}^k\lambda_j\alpha_j$, for some $\lambda_j>0$.
Given such $\LL$, we will regard $\N^k$ as endowed with the total ordering defined by the lexicographic ordering of the $(k+1)$-tuples $(\LL(\alpha),\alpha_k,\dots,\alpha_1)$. The \emph{support} of a non-zero $F=\sum_{\alpha\in\N^k}F_\alpha z^\alpha\in A\{x\}$ is defined as \,$\supp{F}=\{\alpha\in\N^k:F_\alpha\neq0\}$.
The \emph{initial exponent} of $F$, denoted $\inexp_\LL{F}$, is the minimum (with respect to the above total ordering) over all $\alpha\in\supp{F}$. 
Similarly,
\[
\supp{F(0)}=\{\alpha\in\N^k:F_\alpha(0)\neq0\}\quad\mathrm{and}\quad\inexp_\LL{F(0)}={\min}_\LL\{\alpha\in\supp{F(0)}\}\,,
\]
for the evaluated series. We have $\supp{F(0)}\subset\supp{F}$, and hence $\inexp_\LL{F}\leq\inexp_\LL{F(0)}$.
We will write simply $\inexp{F}$ and $\inexp{F(0)}$ instead of $\inexp_\LL{F}$ and $\inexp_\LL{F(0)}$, when $\LL(\alpha)=|\alpha|=\alpha_1+\dots+\alpha_k$.
\medskip

We now recall Hironaka's division algorithm.

\begin{theorem}[{\cite[Thm.\,3.1,\,3.4]{BM1}}]
\label{thm:Hir-div}
Let $\LL$ be any positive linear form on $\K^k$.
Let $G_1,\dots,G_t\in A\{z\}$, and let $\alpha^i\coloneqq\exp_\LL{G_i(0)}$, $1\leq i\leq t$. Then, for every $F\in A\{z\}$, there exist unique $Q_1,\dots,Q_t,R\in A\{z\}$ such that
\begin{gather}
\label{eq:remainder}
F=\sum_{i=1}^tQ_iG_i+R,\\
\notag
\alpha^i+\supp{Q_i}\subset\Delta_i,\ 1\leq i\leq t,\quad\mathrm{and}\quad\supp{R}\subset\Delta,
\end{gather}
where
\[
\Delta_1\coloneqq\alpha^1+\N^k,\qquad\Delta_i\coloneqq(\alpha^i+\N^k)\setminus\bigcup_{j=1}^{i-1}\Delta_j\quad\mathrm{\ for\ }i\geq2,
\]
and $\Delta\coloneqq\N^k\setminus\bigcup_{i=1}^t\Delta_i$.
\end{theorem}

The \emph{diagram of initial exponents} of an ideal $I$ in $A\{z\}$, relative to $\LL$, is defined as
\[
\NN_\LL(I)=\{\inexp_\LL{F}:F\in I\setminus\{0\}\}.
\]
Similarly, for the evaluated ideal $I(0)$, we set
\[
\NN_\LL(I(0))=\{\inexp_\LL{F(0)}:F\in I, F(0)\neq0\}.
\]
We will write $\NN(I)$ and $\NN(I(0))$ instead of $\NN_\LL(I)$ and $\NN_\LL(I(0))$, when $\LL(\alpha)=|\alpha|=\alpha_1+\dots+\alpha_k$.

Note that every diagram $\NN_\LL(I)$ satisfies the equality $\NN_\LL(I)+\N^k=\NN_\LL(I)$. (Indeed, for $\alpha\in\NN_\LL(I)$ and $\gamma\in\N^k$, one can choose $F\in I$ with $\inexp_\LL{F}=\alpha$; then $z^\gamma F\in I$, hence $\alpha+\gamma=\inexp_\LL(z^\gamma F)$ is in $\NN_\LL(I)$.)

\begin{remark}
\label{rem:vertices}
It is easy to show that, for every ideal $I$ in $A\{z\}$ and for every positive linear form $\LL$, there exists a unique smallest (finite) set $V_\LL(I)\subset\NN_\LL(I)$ such that $V_\LL(I)+\N^k=\NN_\LL(I)$ (see, e.g., \cite[Lem.\,3.8]{BM1}). The elements of $V_\LL(I)$ are called the \emph{vertices} of the diagram $\NN_\LL(I)$.
\end{remark}

\begin{corollary}[{\cite[Cor.\,3.9]{BM1}}]
\label{cor:Hir-diagram}
Let $\LL$ be any positive linear form on $\K^k$. Let $I$ be an ideal in $\K\{z\}$, and let $\alpha^1,\dots,\alpha^t\in\N^k$ be the vertices of the diagram $\NN_\LL(I)$. Choose $G_i\in I$ such that $\inexp_\LL{G_i}=\alpha^i$, $1\leq i\leq t$, and let $\{\Delta_i,\Delta\}$ denote the partition of $\N^k$ determined by the $\alpha^i$, as above. Then, $\NN_\LL(I)=\bigcup_{i=1}^t\Delta_i$ and the $G_i$ generate the ideal $I$.
\end{corollary}

\begin{proof}
The equality $\NN_\LL(I)=\bigcup_{i=1}^t\Delta_i$ follows immediately from Remark~\ref{rem:vertices}.
According to Theorem~\ref{thm:Hir-div}, any $F\in\K\{z\}$ can be written as $F=\sum_{i=1}^tQ_iG_i+R_F$, where $\supp{R_F}\subset\Delta$.
Therefore, $F\in I$ if and only if $R_F\in I$. But $\supp{R_F}\subset\Delta=\N^k\setminus\NN_\LL(I)$, hence $R_F\in I$ if and only if $R_F=0$.
\end{proof}

The remainder of this section will be concerned with the algebraic notion of flatness. Recall that a module $M$ over a Noetherian ring $A$ is called \emph{flat} when, for every exact sequence
\[
0\to N'\to N\to N''\to0
\]
of $A$-modules, the sequence
\[
0\to N'\otimes_A\!M\to N\otimes_A\!M\to N''\otimes_A\!M\to0
\]
is also exact.
The following result of Hironaka expresses flatness in terms of his division algorithm.

\begin{theorem}[{\cite[Thm.\,7.9]{BM1}}]
\label{thm:Hir-flat}
Let $I$ be an ideal in $A\{z\}$. Let $\LL$ be any positive linear form on $\K^k$, and let $\alpha^1,\dots,\alpha^t$ be the vertices of $\NN_\LL(I(0))$. Let $G_1,\dots,G_t\in I$ be such that $\inexp_\LL{G_i(0)}=\alpha^i$, $1\leq i\leq t$. Then, the following are equivalent:
\begin{itemize}
\item[(i)] $A\{z\}/I$ is flat as an $A$-module
\item[(ii)] For any $F\in I$, the remainder of $F$ after division~\eqref{eq:remainder} by $G_1,\dots,G_t$ is zero.
\end{itemize}
\end{theorem}
\medskip

Let now $x=(x_1,\dots,x_n)$, $n\geq2$. Fix $k\in\{1,\dots,n-1\}$. To simplify notation, let $x_{[k]}$ denote variables $(x_1,\dots,x_k)$, and $\tx$ the remaining $(x_{k+1},\dots,x_n)$. In what follows, we will regard elements of $\K\{x\}$ either as power series in all the variables $x$ with coefficients in $\K$, written $F=\sum_{\beta\in\N^n}f_\beta x^\beta$, $f_\beta\in\K$, or as power series in variables $x_{[k]}$ with coefficients in $\K\{\tx\}$, written $F=\sum_{\alpha\in\N^k}F_\alpha(\tx)x_{[k]}^\alpha$, $F_\alpha(\tx)\in\K\{\tx\}$.

For $F\in\K\{x\}$, we will denote by $F(0)$ the series with variables $\tx$ evaluated at $0$. That is, if $F=\sum_{\alpha\in\N^k}F_\alpha(\tx)x_{[k]}^\alpha$ then $F(0)=\sum_{\alpha\in\N^k}F_\alpha(0)x_{[k]}^\alpha\in\K\{x_{[k]}\}$. Equivalently, if $F=\sum_{\beta\in\N^n}f_\beta x^\beta$ then
\[
F(0)=\sum_{\beta\in\N^k\times\{0\}^{n-k}}f_\beta x^\beta\in\K\{x_{[k]}\}
\]
(i.e., the sum is over those $\beta=(\beta_1,\dots,\beta_n)$ for which $\beta_{k+1}=\dots=\beta_n=0$).

To avoid confusion, for $F\in\K\{x\}$, we will denote its support as an element of $\K\{\tx\}\{x_{[k]}\}$ by $\supp_{\N^k}F$, and the support as an element of $\K\{x\}$ as $\supp_{\N^n}F$. That is,
\[
\supp_{\N^k}F=\{\alpha\in\N^k:F_\alpha(\tx)\neq0\}\quad\mathrm{and}\quad\supp_{\N^n}F=\{\beta\in\N^n:f_\beta\neq0\}.
\]
Note that a positive linear form $\LL(\beta)=\sum_{i=1}^n\lambda_i\beta_i$ on $\K^n$ gives rise to a positive form $\sum_{i=1}^k\lambda_i\beta_i$ on $\K^k$. By a slight abuse of notation, we will denote the latter also by $\LL$.

\begin{proposition}
\label{prop:flatness-diagram}
Let $I$ be an ideal in $\K\{x\}$, let $1\leq k<n$, and let $\tx$ denote the variables $(x_{k+1},\dots,x_n)$.
\begin{itemize} 
\item[(i)] If there exist a positive linear form $\LL$ on $\K^n$ and a set $D\subset\N^k$ such that $\NN_\LL(I)=D\times\N^{n-k}$\!, then $\K\{x\}/I$ is a flat $\K\{\tx\}$-module.
\item[(ii)] If $\K\{x\}/I$ is a flat $\K\{\tx\}$-module, then there exist $l_0\in\N$ and a set $D\subset\N^k$ such that, for all $l\geq l_0$, the diagram $\NN_\LL(I)$ with respect to the linear form $\displaystyle{\LL(\beta)=\sum_{i=1}^k\beta_i+\sum_{j=k+1}^n\!l\beta_j}$ satisfies $\NN_\LL(I)=D\times\N^{n-k}$.
\end{itemize}
\end{proposition}

\begin{remark}
\label{rem:base}
Note that if $\LL$ is such that $\NN_\LL(I)=D\times\N^{n-k}$ for some $D\subset\N^k$, then necessarily $D=\NN_\LL(I(0))$.
\end{remark}

\begin{proof}
For the proof of (i), we will need the following well known corollary to the classical ``local criterion for flatness'' (see, e.g., \cite[Cor.\,7.6]{BM1}): \emph{$\K\{x\}/I$ is flat as a $\K\{\tx\}$-module if and only if \,$I\cap(\tx)\K\{x\}\subset(\tx)\!\cdot\!I$}.

Suppose that $F\in I\cap(\tx)\K\{x\}$. Then, $F(0)=0$. Let $\beta^1,\dots,\beta^t\in\N^k\times\{0\}^{n-k}$ be the vertices of $\NN_\LL(I)$, and let $G_1,\dots,G_t\in I$ be such that $\inexp_\LL{G_i}=\beta^i$, $1\leq i\leq t$. By Theorem~\ref{thm:Hir-div} and Corollary~\ref{cor:Hir-diagram}, there are $Q_1,\dots,Q_t\in\K\{x\}$ such that $F=\sum_{i=1}^tQ_iG_i$ and the sets $\beta^i+\supp_{\N^n}Q_i$ are pairwise disjoint. 

Write $\beta^i=(\alpha^i,0)$, where $\alpha^i\in\N^k$, $1\leq i\leq t$.
It follows that the sets $\alpha^i+\supp_{\N^k}Q_i(0)$ in $\N^k$ are also pairwise disjoint, and hence the initial exponents $\inexp_\LL{Q_i(0)G_i(0)}=\inexp_\LL{Q_i(0)}+\alpha^i$ are pairwise distinct. On the other hand, $\sum_{i=1}^tQ_i(0)G_i(0)=F(0)=0$. This is only possible if $Q_i(0)=0$ for all $i$. In other words, $Q_i\in(\tx)\K\{x\}$. Hence $F=\sum_{i=1}^tQ_iG_i$ is in $(\tx)\!\cdot\!I$, which proves (i).
\smallskip

Suppose now that $\K\{x\}/I$ is $\K\{\tx\}$-flat. Let $\lambda(\alpha)=|\alpha|$ for $\alpha\in\N^k$, and let $\alpha^1,\dots,\alpha^t$ be the vertices of $\NN(I(0))=\NN_\lambda(I(0))$. Let $\{\Delta_i,\Delta\}$ be the partition of $\N^k$ determined by the $\alpha^i$ as in Theorem~\ref{thm:Hir-div}. Let $l_0=1+\max\{|\alpha^i|:i=1,\dots,t\}$, let $l\geq l_0$ be arbitrary, and set
\[
\LL(\beta)\coloneqq\sum_{i=1}^k\beta_i+\sum_{j=k+1}^n\!l\beta_j\,.
\]
Set $\beta^i\coloneqq(\alpha^i,0)=(\alpha^i_1,\dots,\alpha^i_k,0,\dots,0)\in\N^n$, $1\leq i\leq t$. We will show that the vertices of $\NN_\LL(I)$ are precisely $\{\beta^1,\dots,\beta^t\}$.

Let $G_1,\dots,G_t\in I$ be such that $\inexp{G_i(0)}=\alpha^i$, $1\leq i\leq t$. Write $G_i=\sum_{\alpha\in\N^k}G_{i,\alpha}x_{[k]}^\alpha$, where $G_{i,\alpha}=\sum_{\gamma\in\N^{n-k}}g_{i,\alpha,\gamma}\tx^\gamma$. For every $i$, there are at most finitely many $\alpha\in\supp_{\N^k}G_i$ with $\lambda(\alpha)<\lambda(\alpha^i)$. For each such $\alpha$, by the choice of $\alpha^i$, we have $\nu(G_{i,\alpha})\geq1$, where for a non-zero $F\in\K\{\tx\}$, $\nu(F)=\max\{r:F\in(\tx)^r\}$. Therefore, for each such $\alpha$ and for every non-zero term $g_{i,\alpha,\gamma}\tx^\gamma$ of $G_{i,\alpha}$, we have $|\gamma|\geq1$, and hence $\LL((\alpha,\gamma))\geq l_0>\LL(\beta^i)$.
It follows that, with respect to the total ordering in $\N^n$ induced by $\LL$, we have
\[
\inexp_\LL(G_i)=\beta^i,\ \ 1\leq i\leq t.
\]
Pick $F\in I$. By Theorem~\ref{thm:Hir-flat}, there are $Q_1,\dots,Q_t\in\K\{x\}$ such that $F=\sum_{i=1}^tQ_iG_i$ and $\alpha^i+\supp_{\N^k}Q_i\subset\Delta_i$, $1\leq i\leq t$. Then, $\beta^i+\supp_{\N^n}Q_i\subset\Delta_i\times\N^{n-k}$; in particular, $\beta^i+\inexp_\LL{Q_i}\in\Delta_i\times\N^{n-k}$, $1\leq i\leq t$. It thus follows that the $\inexp_\LL(Q_iG_i)=\beta^i+\inexp_\LL{Q_i}$ lie in pairwise disjoint regions, and hence are pairwise distinct. Consequently, $\inexp_\LL{F}=\min_\LL\{\inexp_\LL(Q_iG_i):i=1,\dots,t\}$ belongs to $\NN(I(0))\times\N^{n-k}$, which completes the proof.
\end{proof}
\medskip

%%%%%%%%%%%%%%%%%%%%%%%%%%%%%%%%%%%%%%%%%%%%%%%%%%%
%%%%%%%%%%%%%%%%%%%%% Section %%%%%%%%%%%%%%%%%%%%%
%%%%%%%%%%%%%%%%%%%%%%%%%%%%%%%%%%%%%%%%%%%%%%%%%%%
\section{Standard bases and Becker's s-series criterion}
\label{sec:s-series}

Let again $\LL(\beta)=\sum_{j=1}^n\lambda_j\beta_j$ be a positive linear form on $\K^n$, and let $\N^n$ be given the total ordering defined by the lexicographic ordering of the $(n+1)$-tuples $(\LL(\beta),\beta_n,\dots,\beta_1)$.
For $F\in\K\{x\}$, let as before $\inexp_\LL{F}={\min}_\LL\{\beta\in\supp{F}\}$ denote the initial exponent of $F$ relative to $\LL$.

Let $I$ be an ideal in $\K\{x\}$. Following Becker \cite{Be1}, we will say that a collection $S=\{G_1,\dots,G_t\}\subset I$ forms a \emph{standard basis} of $I$ (relative to $\LL$), when for every $F\in I$ there exists $i\in\{1,\dots,t\}$ such that $\inexp_\LL{F}\in\inexp_\LL{G_i}+\N^n$.

\begin{remark}
\label{rem:st-bases}
\begin{itemize}
\item[(1)] It follows directly from definition that every standard basis $S$ of $I$ relative to $\LL$ contains representatives of all the vertices of the diagram $\NN_\LL(I)$ (that is, for every vertex $\beta^i$ of $\NN_\LL(I)$ there exists $G_i\in S$ with $\beta^i=\inexp_\LL{G_i}$). Hence, by Corollary~\ref{cor:Hir-diagram}, every standard basis of $I$ is a set of generators of $I$.
\item[(2)] Note that the term ``standard basis'' in most of modern literature refers to a collection defined by more restrictive conditions than the one above (see, e.g., \cite[Cor.\,3.9]{BM1} or \cite[Cor.\,3.19]{BM2}). In particular, our standard basis is not unique and may contain elements which do not represent vertices of the diagram.
\end{itemize}
\end{remark}

For any $F=\sum_\beta f_\beta x^\beta$ and $G=\sum_\beta g_\beta x^\beta$ in $\K\{x\}$, one defines their \emph{s-series} $S(F,G)$ with respect to $\LL$ as follows: If $\beta_F=\inexp_\LL{F}$, $\beta_G=\inexp_\LL{G}$, and $x^\gamma=\mathrm{lcm}(x^{\beta_F},x^{\beta_G})$, then
\[
S(F,G)\coloneqq g_{\beta_G}x^{\gamma-\beta_F}\!\cdot F - f_{\beta_F}x^{\gamma-\beta_G}\!\cdot G\,.
\]
Given $G_1,\dots,G_t,F\in\K\{x\}$, we say that $F$ has a \emph{standard representation} in terms of $\{G_1,\dots,G_t\}$ with respect to $\LL$, when there exist $Q_1,\dots,Q_t\in\K\{x\}$ such that
\[
F=\sum_{i=1}^tQ_iG_i\qquad\mathrm{and}\qquad \inexp_\LL{F}\leq\min\{\inexp_\LL(Q_iG_i):i=1,\dots,t\}\,.
\]
Here, we adopt a convention that $\inexp_\LL{F}<\inexp_\LL{0}$, for any $\LL$ and any non-zero $F$.
\smallskip

The following s-series criterion of Becker will be our main tool in establishing Hilbert-Samuel equisingularity.

\begin{theorem}[{\cite[Thm.\,4.1]{Be1}}]
\label{thm:Becker}
Let $S$ be a finite subset of $\K\{x\}$. Then, $S$ is a standard basis (relative to $\LL$) of the ideal it generates if and only if for any $G_1,G_2\in S$ the s-series $S(G_1,G_2)$ has a standard representation in terms of $S$.
\end{theorem}
\medskip

%%%%%%%%%%%%%%%%%%%%%%%%%%%%%%%%%%%%%%%%%%%%%%%%%%%
%%%%%%%%%%%%%%%%%%%%% Section %%%%%%%%%%%%%%%%%%%%%
%%%%%%%%%%%%%%%%%%%%%%%%%%%%%%%%%%%%%%%%%%%%%%%%%%%
\section{Nested parametrized algebraic approximation}
\label{sec:APP}

Let $x=(x_1,\dots,x_n)$, $y=(y_1,\dots,y_m)$, and let $\K\!\left\langle{x}\right\rangle$ denote the ring of algebraic power series in $x$.
Recall that a convergent power series $F\in\K\{x\}$ is called an \emph{algebraic power series} when $F$ is algebraic over the ring of polynomials $\K[x]$.

The following nested variant of P{\l}oski's parametrized approximation theorem \cite{Plo} is due to Bilski, Parusi{\'n}ski and Rond \cite{BPR}. The result itself follows from Spivakovsky's nested approximation \cite[Thm.\,11.4]{Spi}, which in turn is a corollary of the N{\'e}ron-Popescu Desingularization \cite{Pop}.

\begin{theorem}[{\cite[Thm.\,2.1]{BPR}}]
\label{thm:BPR}
Let $f(x,y)=(f_1(x,y),\dots,f_p(x,y))\in\K\!\left\langle{x}\right\rangle\![y]^p$ and let $\bar{y}(x)=(\bar{y}_1(x),\dots,\bar{y}_m(x))\in\K\{x\}^m$ be such that
\[
f(x,\bar{y}(x))=0.
\]
Suppose that $\bar{y}_i(x)$ depends only on variables $(x_1,\dots,x_{\sigma(i)})$, where $\{i\mapsto\sigma(i)\}$ is an increasing function. Then, there exist a new set of variables $z=(z_1,\dots,z_s)$, an increasing function $\tau$, an $m$-tuple of algebraic power series $\hat{y}(x,z)\in\K\left\langle{x,z}\right\rangle^m$ such that
\[
f(x,\hat{y}(x,z))=0,
\]
and for every $i$,
\[
\hat{y}_i(x,z)\in\K\left\langle{x_1,\dots,x_{\sigma(i)},z_1,\dots,z_{\tau(i)}}\right\rangle\,,
\]
as well as convergent power series $\bar{z}_i(x)\in\K\{x\}$ vanishing at $0$ such that $\bar{z}_1(x),\dots$, $\bar{z}_{\tau(i)}(x)$ depend only on $(x_1,\dots,x_{\sigma(i)})$ and
\[
\bar{y}(x)=\hat{y}(x,\bar{z}(x))\,.
\]
\end{theorem}

The classical Algebraic Artin Approximation follows immediately from the above.

\begin{theorem}[{\cite[Thm.\,1.10]{Ar2}}]
\label{thm:Artin}
Let $f(x,y)\in\K\!\left\langle{x}\right\rangle\![y]^p$ and let $\bar{y}(x)\in\K\{x\}^m$ be such that
\[
f(x,\bar{y}(x))=0.
\]
Then, for any $c\in\N$, there exists an $m$-tuple of algebraic power series $\hat{y}(x)\in\K\left\langle{x}\right\rangle^m$ such that
\[
f(x,\hat{y}(x))=0,
\]
and $\hat{y}$ coincides with $\bar{y}$ up to degree $c$, that is, $\bar{y}(x)-\hat{y}(x)\in(x)^{c+1}$.
\end{theorem}

\begin{proof}
Let $\hat{y}(x,z)\in\K\left\langle{x,z}\right\rangle^m$ and the $\bar{z}_i(x)\in\K\{x\}$ be as in Theorem~\ref{thm:BPR}. Then, for a given $c\in\N$, the $m$-tuple $\hat{y}(x,j^c(\bar{z}(x))$ has the desired properties, where $j^cF$ denotes the $c$-jet of $F\in\K\{x\}$.
\end{proof}
\medskip

%%%%%%%%%%%%%%%%%%%%%%%%%%%%%%%%%%%%%%%%%%%%%%%%%%%
%%%%%%%%%%%%%%%%%%%%% Section %%%%%%%%%%%%%%%%%%%%%
%%%%%%%%%%%%%%%%%%%%%%%%%%%%%%%%%%%%%%%%%%%%%%%%%%%
\section{Reduction of the maximal ideal}
\label{sec:reduction}

Let, as before, $x=(x_1,\dots,x_n)$, and for $k<n$, let $x_{[k]}=(x_1,\dots,x_k)$ and $\tx=(x_{k+1},\dots,x_n)$.
Let $\mm$ denote the maximal ideal of $\K\{x\}$.
The purpose of this section is to find a suitable reduction (in the sense of Northcott-Rees \cite{NR}) of the maximal ideal $\mm/I$ in $\K\{x\}/I$, for a given ideal $I$ in $\K\{x\}$.

The following proposition is a simple consequence of the Weierstrass Preparation Theorem (see, e.g., \cite[Ch.\,III, Prop.\,2]{Nar}).

\begin{proposition}
\label{prop:Narasimhan}
Let $I$ be a proper ideal in $\K\{x\}$. After a generic linear change of coordinates in $\K^n$, there exists $k\in\{0,\dots,n-1\}$ such that the natural homomorphism $\K\{\tx\}\to\K\{x\}/I$ is injective and makes $\K\{x\}/I$ into a finite $\K\{\tx\}$-module.
\end{proposition}

\begin{lemma}
\label{lem:tangent-cone}
Let $I$ be an ideal in $\K\{x\}$ with $\dim\K\{x\}/I=n-k$. Then, after a generic linear change of coordinates in $\K^n$, there is, for every $j=1,\dots,k$, a distinguished polynomial $P_j\in\K\{\tx\}[t]$ of degree $d_j$ such that $P_j(x_j,\tx)\in I\cap\mm^{d_j}$, where $\tx=(x_{k+1},\dots,x_n)$.
\end{lemma}

\begin{proof}
By Proposition~\ref{prop:Narasimhan}, after a generic linear change of coordinates in $\K^n$, there exists $k'\leq{n-1}$ and an injective homomorphism $\K\{\tx\}\to\K\{x\}/I$ such that $\K\{x\}/I$ is a finite $\K\{\tx\}$-module, where $\tx=(x_{k'+1},\dots,x_n)$. Since for a finite injective homomorphism of Noetherian rings $A\to R$ we have $\dim{R}=\dim{A}$, it follows that $k'=k$.

Suppose first that $\K=\C$. Let $(X,0)$ be the germ of an analytic space at $0$ in $\C^n$ defined by $\OO_{X,0}=\C\{x\}/I$. Further, let $C(X,0)$ denote the tangent cone to $(X,0)$, in the sense of Whitney \cite{Wh}. Then, $\dim_0{C(X,0)}=\dim_0X=n-k$, and after another generic linear change of coordinates if needed, we may assume that $C(X,0)$ has a proper and surjective projection onto (an open neighbourhood of $0$ in) $\C^{n-k}$ spanned by the variables $\tx$. Finiteness of $\C\{x\}/I$ as a $\C\{\tx\}$-module implies that the images of $x_1,\dots,x_k$ in $\C\{x\}/I$ are integral over $\C\{\tx\}$. Hence, for every $j=1,\dots,k$, there exist $d_j\in\Z_+$ and a distinguished polynomial $P_j\in\C\{\tx\}[t]$ of degree $d_j$, such that $P_j(x_j,\tx)\in I$. Write $\displaystyle{P_j(x_j,\tx)=x_j^{d_j}+\sum_{r=1}^{d_j}a^j_r(\tx)x_j^{d_j-r}}$, $j=1,\dots,k$.

Fix $j\in\{1,\dots,k\}$. Let $LF(P_j)$ denote the leading form of $P_j$ (i.e., the homogeneous polynomial consisting of the terms of $P_j$ of lowest degree). By Whitney's theory of tangent cones~\cite{Wh}, $C(X,0)$ is the set of common zeroes of leading forms $LF(F)$ for all $F$ vanishing on $(X,0)$. In particular, $C(X,0)\subset LF(P_j)^{-1}(0)$.

To prove the lemma, it now suffices to show that $x_j^{d_j}$ is among the terms of $LF(P_j)$.
We argue by induction on $n-k$, the number of variables $\tx$.
If $n-k=1$, then $\tx$ is a single variable $x_n$. If $x_j^{d_j}$ were not among the terms of $LF(P_j)$ then $x_n$ would divide $LF(P_j)$, and so the image of $C(X,0)$ under the projection to $\C^{n-k}$ would be $\{0\}=LF(P_j)^{-1}(0)\cap\{x_{[k]}=0\}$, contradicting the surjectivity.

Suppose then that $n-k\geq2$, and consider $\widetilde{P}_j:=P_j(x_j,x_{k+1},\dots,x_{n-1},0)$. Then, $\widetilde{P}_j$ vanishes on $\widetilde{X}:=X\cap\{x_n=0\}$, and hence $LF(\widetilde{P}_j)$ vanishes on $C(\widetilde{X},0)$. Since $C(\widetilde{X},0)$ has a surjective projection onto (an open neighbourhood of $0$ in) $\C^{n-k-1}$, then, by induction, $x_j^{d_j}$ is among the terms of $LF(\widetilde{P}_j)$. If $x_j^{d_j}$ were not among the terms of $LF(P_j)$, then we would have $\deg{LF(P_j)}<\deg{LF(\widetilde{P}_j)}=d_j$. Hence, $x_n$ would divide $LF(P_j)$, and so the image of $C(X,0)$ under projection to $\C^{n-k}$ would be contained in the hypersurface $\{x_n=0\}$. This contradiction completes the proof in case $\K=\C$.

If $\K=\R$, the result follows by applying the above argument to the complexification $X^\C$ of $X$. Note that the linear changes of coordinates required at the beginning may be taken with integral coefficients, and hence the distinguished polynomials $P_j$ will have real coefficients.
\end{proof}

Let $P_j(x_j,\tx)=x_j^{d_j}+\sum_{r=1}^{d_j}a^j_r(\tx)x_j^{d_j-r}$, $j=1,\dots,k$, be as above. Set
\[
d:=\sum_{j=1}^k(d_j-1).
\]

\begin{corollary}
\label{cor:reduction}
We have $(x_{[k]})^{d+1}\subset I+(\tx)\!\cdot\!\mm^d$, as ideals in $\K\{x\}$.
\end{corollary}

\begin{proof}
Indeed, for any monomial $x_1^{\beta_1}\!\cdots x_k^{\beta_k}\in(x_{[k]})^{d+1}$, there exists $j$ such that $\beta_j\geq d_j$.
By Lemma~\ref{lem:tangent-cone}, $x_j^{d_j}=P_j(x_j,\tx)-\sum_{r=1}^{d_j}a^j_r(\tx)x_j^{d_j-r}$ is an element of $I+(\tx)\cap\mm^{d_j}=I+(\tx)\cdot\mm^{d_j-1}$. Consequently, $x_1^{\beta_1}\dots x_k^{\beta_k}\in I+(\tx)\cdot\mm^N$, where $N=\beta_1+\dots+(\beta_j-1)+\dots+\beta_k\geq d$.
\end{proof}

\begin{remark}
\label{rem:reduction}
The above corollary implies that $I+(\tx)/I$ is a \emph{reduction} (with exponent $d$) of the maximal ideal $\mm/I$ in $\K\{x\}/I$, in the sense of Northcott--Rees~\cite{NR}. Indeed, one trivially has $I+(\tx)\subset I+\mm$, and by above, $I+\mm^{d+1}\,\subset\,I+(\tx)\cdot\mm^d$. It follows that $I+\mm^{d+1}\,=\,I+(\tx)\cdot\mm^d$, hence by induction
\begin{equation}
\label{eq:red}
I+\mm^{d+m}\,=\,I+(\tx)^m\mm^d,\quad\mathrm{for\ any\ }m\geq1.
\end{equation}
\end{remark}
\medskip

%%%%%%%%%%%%%%%%%%%%%%%%%%%%%%%%%%%%%%%%%%%%%%%%%%%
%%%%%%%%%%%%%%%%%%%%% Section %%%%%%%%%%%%%%%%%%%%%
%%%%%%%%%%%%%%%%%%%%%%%%%%%%%%%%%%%%%%%%%%%%%%%%%%%
\section{Approximation of ideals and diagrams}
\label{sec:approx}

Let, as before, $\LL(\beta)=\sum_{j=1}^n\lambda_j\beta_j$ be a positive linear form on $\K^n$.
For such $\LL$ and $\mu\in\N$, define $\nn_{\LL,\mu}$ to be the ideal in $\K\{x\}$ generated by all the monomials $x^\beta=x_1^{\beta_1}\dots x_n^{\beta_n}$ with $\LL(\beta)\geq\mu$. (Note that, by positivity of the linear form $\LL$, the ideals $\nn_{\LL,\mu}$ are $\mm$-primary for every $\mu$. Moreover, for $\LL(\beta)=|\beta|$ we have $\nn_{\LL,\mu}=\mm^\mu$.)
For a natural number $\mu\in\N$ and a power series $F\in\K\{x\}$, the \emph{$\mu$-jet of $F$ with respect to $\LL$}, denoted $j_\LL^\mu(F)$, is the image of $F$ under the canonical epimorphism $\K\{x\}\to\K\{x\}/\nn_{\LL,\mu+1}$. We will write $j^\mu(F)$ for $\mu$-jets with respect to $\LL(\beta)=|\beta|=\beta_1+\dots+\beta_n$.

\begin{remark}
\label{rem:jet-exp}
Given a power series $F\in\K\{x\}$, suppose that $\mu\geq\LL(\inexp_\LL(F))$.
Then, $\inexp_\LL(F)=\inexp_\LL(G)$ for every $G\in\K\{x\}$ with $j_\LL^\mu(G)=j_\LL^\mu(F)$.
\end{remark}

The following lemma expresses the Hilbert-Samuel function of an ideal in terms of its diagram of initial exponents.

\begin{lemma}
\label{lem:HS-diagram-complement}
Let $\lambda_1,\dots,\lambda_n>0$ be arbitrary, and let $\LL(\beta)=\sum_{j=1}^n\lambda_j\beta_j$. Then, for any ideal $I$ in $\K\{x\}$ and for every $\eta\geq1$,
\[
\#\{\beta\in\N^n\setminus\NN_\LL(I):\LL(\beta)\leq\eta\}\ =\ \dim_\K\frac{\K\{x\}}{I+\nn_{\LL,\eta+1}}\,,
\]
where the dimension on the right side is in the sense of $\K$-vector spaces.
In particular, the Hilbert-Samuel function $H_I$ of \,$\K\{x\}/I$ satisfies
\[
H_I(\eta)=\#\{\beta\in\N^n\setminus\NN(I):|\beta|\leq\eta\},\quad\mathrm{for\ all\ }\eta\geq1\,.
\]
\end{lemma}

\begin{proof}
Fix $\eta\geq1$.
Suppose that $F\in\K\{x\}$ satisfies $\supp(F)\subset\{\beta\in\N^n\setminus\NN_\LL(I):\LL(\beta)\leq\eta\}$ and pick $G\in\nn_{\LL,\eta+1}$. Then, $\inexp_\LL(F+G)=\inexp_\LL(F)$, by Remark~\ref{rem:jet-exp}, and hence $\inexp_\LL(F+G)\notin\NN_\LL(I)$. It follows that $F+G\notin I$, and so $F\notin I+\nn_{\LL,\eta+1}$. This proves that the set of monomials $\{x^\beta: \beta\in\N^n\setminus\NN_\LL(I),\LL(\beta)\leq\eta\}$ is linearly independent in $\K\{x\}/(I+\nn_{\LL,\eta+1})$, whence
\[
\dim_\K\frac{\K\{x\}}{I+\nn_{\LL,\eta+1}}\ \geq\ \#\{\beta\in\N^n\setminus\NN_\LL(I):\LL(\beta)\leq\eta\}\,.
\]
Conversely, suppose that $F\notin I+\nn_{\LL,\eta+1}$. Let $G_1,\dots,G_t\in I$ be representatives of the vertices of $\NN_\LL(I)$ and let $F=\sum_{i=1}^tQ_iG_i+R$ be the unique Hironaka division of $F$ by the $G_i$ in $\K\{x\}$, relative to $\LL$.
Now, if $R\in\nn_{\LL,\eta+1}$ then $F\in I+\nn_{\LL,\eta+1}$; a contradiction. Therefore, we have $R=R_1+R_2$, with $R_2\in\nn_{\LL,\eta+1}$, $R_1\neq0$, and $\supp(R_1)\subset\{\beta\in\N^n\setminus\NN_\LL(I):\LL(\beta)\leq\eta\}$ (cf. Theorem~\ref{thm:Hir-div}). Then, $F-R_1=\sum_{i=1}^tQ_iG_i+R_2$ is in $I+\nn_{\LL,\eta+1}$, which shows that $F$ and $R_1$ represent the same element of $\K\{x\}/(I+\nn_{\LL,\eta+1})$. Thus,
\[
\dim_\K\frac{\K\{x\}}{I+\nn_{\LL,\eta+1}}\ \leq\ \#\{\beta\in\N^n\setminus\NN_\LL(I):\LL(\beta)\leq\eta\}\,.
\]
The last claim of the lemma now follows from the definition of the Hilbert-Samuel function as $H_I(\eta)=\dim_\K\K\{x\}/(I+\mm^{\eta+1})$.
\end{proof}
\medskip

\begin{definition}
\label{def:mu-nbhd}
For an ideal $I=(F_1,\dots,F_s)\cdot\K\{x\}$, a positive linear form $\LL$ and $\mu\geq1$, we define the family of ideals $U_\LL^\mu(I)$ (or, more precisely, $U_\LL^\mu(F_1,\dots,F_s)$) as
\[
U_\LL^\mu(I)=\{(G_1,\dots,G_s)\cdot\K\{x\}:j_\LL^\mu(G_i)=j_\LL^\mu(F_i), 1\leq i\leq s\}.
\]
We will write simply $U^\mu(I)$ for $U_\LL^\mu(I)$, when $\LL(\beta)=|\beta|$.
\end{definition}

The following lemma shows that the reduction of the maximal ideal in $\K\{x\}/I$ is preserved by its sufficiently close Taylor approximations.

\begin{lemma}
\label{lem:red-approx}
Let $I=(F_1,\dots,F_s)$ be an ideal in $\K\{x\}$ with $\dim\K\{x\}/I=n-k$. Then, after a generic linear change of coordinates in $\K^n$, there exists $\mu_0$ such that, for every $\mu\geq\mu_0$ and $I_\mu\in U^\mu(I)$, we have
\begin{equation}
\label{eq:red-approx}
I_\mu+\mm^{d+m}\,=\,I_\mu+(\tx)^m\mm^d,\quad\mathrm{for\ any\ }m\geq1,
\end{equation}
where $d$ is the same as in \eqref{eq:red}.
\end{lemma}

\begin{proof}
After a generic linear change of coordinates from Lemma~\ref{lem:tangent-cone}, we may assume that $(x_{[k]})^{d+1}\subset I+(\tx)\!\cdot\!\mm^d$, where $d$ is as in \eqref{eq:red}. Set $\mu_0\coloneqq d+1$. Pick $\mu\geq\mu_0$ and $I_\mu\in U^\mu(I)$.
Then, $I\subset I_\mu+\mm^{d+2}$, and hence $(x_{[k]})^{d+1}\subset I_\mu+\mm^{d+2}+(\tx)\mm^d$. It follows that
\[
I_\mu+\mm^{d+1}\subset I_\mu+\mm^{d+2}+(\tx)\mm^d\subset I_\mu+(\tx)\mm^d+(I_\mu+\mm^{d+1})\mm\,,
\]
hence $I_\mu+\mm^{d+1}\subset I_\mu+(\tx)\mm^d$, by Nakayama's lemma. The claim now follows as in Remark~\ref{rem:reduction}.
\end{proof}

Let us recall now a results from \cite{AS3} describing the connection between the diagram of initial exponents of $I$ and those of its approximations $I_\mu$. We include a short proof for the reader's convenience.

\begin{lemma}[{cf.\,\cite[Lem.\,3.2]{AS3}}]
\label{lem:diagram-up-to-l}
Let $I$ be an ideal in $\K\{x\}$ and let $\LL$ be a positive linear form on $\K^n$. Let $l_0=\max\{\LL(\beta^i):1\leq i\leq t\}$, where $\beta^1,\dots,\beta^t$ are the vertices of the diagram $\NN_\LL(I)$. Then:
\begin{itemize}
\item[(i)] For every $\mu\geq l_0$ and $I_\mu\in U_\LL^\mu(I)$, we have $\NN_\LL(I_\mu)\supset\NN_\LL(I)$.
\item[(ii)] Given $l\geq l_0$, for every $\mu\geq l$ and $I_\mu\in U_\LL^\mu(I)$, we have
\[
\NN_\LL(I_\mu)\cap\{\beta\in\N^n:\LL(\beta)\leq l\}=\NN_\LL(I)\cap\{\beta\in\N^n:\LL(\beta)\leq l\}\,.
\]
\end{itemize}
\end{lemma}

\begin{proof}
Fix $\mu\geq l_0$ and let $G_1,\dots,G_s\in\K\{x\}$ be such that $I_\mu=(G_1,\dots,G_s)$ and $j_\LL^\mu(G_i)=j_\LL^\mu(F_i)$, $1\leq i\leq s$. 
By Remark~\ref{rem:vertices}, for the proof of (i) it suffices to show that the vertices of $\NN_\LL(I)$ are contained in $\NN_\LL(I_\mu)$.
Let then $F\in I$ be a representative of a vertex of $\NN_\LL(I)$ (i.e., $\inexp_\LL(F)\in V_\LL(I)$). We can write $F=\sum_{i=1}^sH_iF_i$, for some $H_i\in\K\{x\}$. Then,
\[
j_\LL^\mu(F)=j_\LL^\mu(\sum_{i=1}^sH_iF_i)=j_\LL^\mu(\sum_{i=1}^sH_i\!\cdot\!j_\LL^\mu{F_i})=j_\LL^\mu(\sum_{i=1}^sH_i\!\cdot\!j_\LL^\mu{G_i})
=j_\LL^\mu(\sum_{i=1}^sH_iG_i)\,,
\]
and hence, by Remark~\ref{rem:jet-exp}, we have $\inexp_\LL(F)=\inexp_\LL(\sum_{i=1}^sH_iG_i)$.
It follows that $\inexp_\LL(F)\in\NN_\LL(I_\mu)$, which proves (i).

For the proof of part (ii), fix $l\geq l_0$. Let $\mu\geq l$ and let $I_\mu=(G_1,\dots,G_s)$ with $j_\LL^\mu(G_i)=j_\LL^\mu(F_i)$, $1\leq i\leq s$. By part (i), it now suffices to show that
\[
\NN_\LL(I_\mu)\cap\{\beta\in\N^n:\LL(\beta)\leq l\}\ \subset\ \NN_\LL(I)\cap\{\beta\in\N^n:\LL(\beta)\leq l\}\,.
\]
Pick $\beta^*\in\N^n\setminus\NN_\LL(I)$ with $\LL(\beta^*)\leq l$. Suppose that $\beta^*\in\NN_\LL(I_\mu)$. Then, one can choose $G\in I_\mu$ with $\inexp_\LL(G)=\beta^*$. Write $G=\sum_{i=1}^sH_iG_i$ for some $H_i\in\K\{x\}$.
We have
\[
j_\LL^\mu(G)=j_\LL^\mu(\sum_{i=1}^sH_iG_i)=j_\LL^\mu(\sum_{i=1}^sH_i\cdot j_\LL^\mu{G_i})=j_\LL^\mu(\sum_{i=1}^sH_i\cdot j_\LL^\mu{F_i})=j_\LL^\mu(\sum_{i=1}^sH_iF_i)\,,
\]
and since $\mu\geq l\geq\LL(\inexp_\LL(G))$, it follows that $\inexp_\LL(G)=\inexp_\LL(\sum_{i=1}^sH_iF_i)$, by Remark~\ref{rem:jet-exp} again. Therefore $\beta^*\in\NN_\LL(I)$; a contradiction.
\end{proof}

\begin{corollary}
\label{cor:diagram-finite}
Let $I$ be an ideal in $\K\{x\}$ and let $\LL$ be a positive linear form on $\K^n$. Suppose that the complement $\N^n\setminus\NN_\LL(I)$ is finite. Then, there exists $\mu_0\in\N$ such that, for every $\mu\geq\mu_0$ and $I_\mu\in U_\LL^\mu(I)$, we have $\NN_\LL(I_\mu)=\NN_\LL(I)$.
\end{corollary}

\begin{proof}
Let $\beta^1,\dots,\beta^t$ be the vertices of $\NN_\LL(I)$.
Since $\N^n\setminus\NN_\LL(I)$ is finite, there exists $\mu_0\geq\max_i\LL(\beta^i)$ such that
\[
\N^n\setminus\NN_\LL(I)\subset\{\beta\in\N^n:\LL(\beta)\leq\mu_0\}\,.
\]
By Lemma~\ref{lem:diagram-up-to-l} part (i), for every $\mu\geq\mu_0$ and $I_\mu\in U_\LL^\mu(I)$, we have
\[
\N^n\setminus\NN_\LL(I_\mu)\subset\N^n\setminus\NN_\LL(I)
\]
and by part (ii)
\[
(\N^n\setminus\NN_\LL(I_\mu))\cap\{\beta\in\N^n:\LL(\beta)\leq\mu_0\}=(\N^n\setminus\NN_\LL(I))\cap\{\beta\in\N^n:\LL(\beta)\leq\mu_0\}\,.
\]
Thus, $\N^n\setminus\NN_\LL(I_\mu)=\N^n\setminus\NN_\LL(I)$, as required.
\end{proof}

The following proposition is a key tool in the proofs of Theorems~\ref{thm:CI-approx} and~\ref{thm:CM-approx}. It shows that the equality of diagrams of an ideal $I$ and its approximation $I_\mu$ with respect to some ordering on $\N^n$ implies the equality of diagrams with respect to the standard ordering.

\begin{proposition}
\label{prop:HS-stability}
Let $I=(F_1,\dots,F_s)$ be an ideal in $\K\{x\}$ with $\dim\K\{x\}/I=n-k$. Then, after a generic linear change of coordinates in $\K^n$, there exists $\mu_0$ such that, for every $\mu\geq\mu_0$ the following holds: If $I_\mu\in U^\mu(I)$ is such that $I_\mu\cap(\tx)=I_\mu\!\cdot(\tx)$ and $\dim_\K\K\{x\}/I+(\tx)^m=\dim_\K\K\{x\}/I_\mu+(\tx)^m$ for all $m\geq1$, then $H_I=H_{I_\mu}$ (that is, the Hilbert-Samuel functions of $I$ and $I_\mu$ coincide).
\end{proposition}

\begin{proof}
To simplify notation, we shall write $\nn$ for the ideal $(\tx)$ in $\K\{x\}$.
By Lemma~\ref{lem:tangent-cone}, Remark~\ref{rem:reduction} and Lemma~\ref{lem:red-approx}, after a generic linear change of coordinates in $\K^n$, we may assume that there exist positive integers $d$ and $\mu_0$ such that
\begin{equation}
\label{eq:red-I}
I+\mm^{d+m}\,=\,I+\nn^m\mm^d, \mathrm{\ for\ all\ }m\geq1\,,
\end{equation}
and for every $\mu\geq\mu_0$ and $I_\mu\in U^\mu(I)$
\begin{equation}
\label{eq:red-I-mu}
I_\mu+\mm^{d+m}\,=\,I_\mu+\nn^m\mm^d, \mathrm{\ for\ all\ }m\geq1\,.
\end{equation}
By Lemmas~\ref{lem:HS-diagram-complement} and~\ref{lem:diagram-up-to-l}, taking $\mu_0$ sufficiently large, we always have $H_I(\eta)\geq H_{I_\mu}(\eta)$ for all $\eta\geq1$ and $\mu\geq\mu_0$. Therefore, to prove the equality $H_I=H_{I_\mu}$, it suffices to show that $H_I(\eta)=H_{I_\mu}(\eta)$ for all $\eta\geq d$, or equivalently that
\begin{equation}
\label{eq:HS1equal}
\dim_\K\frac{\K\{x\}}{I+\nn^m\mm^d}\ =\ \dim_\K\frac{\K\{x\}}{I_\mu+\nn^m\mm^d}\qquad\mathrm{for\ all\ }m\geq1.
\end{equation}
Fix $\mu\geq\mu_0$ and $I_\mu\in U^\mu(I)$. 
We have, for $m\geq1$, the following exact sequences
\begin{align*}
0\ \to\ \frac{I+\nn^m}{I+\nn^m\mm^d}\ \to\ &\frac{\K\{x\}}{I+\nn^m\mm^d}\ \to\ \frac{\K\{x\}}{I+\nn^m}\ \to\ 0\,,\\
0\ \to\ \frac{I_\mu+\nn^m}{I_\mu+\nn^m\mm^d}\ \to\ &\frac{\K\{x\}}{I_\mu+\nn^m\mm^d}\ \to\ \frac{\K\{x\}}{I_\mu+\nn^m}\ \to\ 0\,.
\end{align*}
By assumption, $\dim_\K\K\{x\}/(I+\nn^m)=\dim_\K\K\{x\}/(I_\mu+\nn^m)$, and hence to prove \eqref{eq:HS1equal}, it suffices to show that
\[
\dim_\K\frac{I+\nn^m}{I+\nn^m\mm^d}\ =\ \dim_\K\frac{I_\mu+\nn^m}{I_\mu+\nn^m\mm^d}\qquad\mathrm{for\ all\ }m\geq1.
\]
Note that
\begin{equation}
\label{eq:iso1}
\frac{I+\nn^m}{I+\nn^m\mm^d}\cong\frac{\nn^m}{(I+\nn^m\mm^d)\cap\nn^m}=\frac{\nn^m}{(I\cap\nn^m)+\nn^m\mm^d}
\end{equation}
and
\begin{equation}
\label{eq:iso2}
\frac{I_\mu+\nn^m}{I_\mu+\nn^m\mm^d}\cong\frac{\nn^m}{(I_\mu+\nn^m\mm^d)\cap\nn^m}=\frac{\nn^m}{(I_\mu\cap\nn^m)+\nn^m\mm^d}\,.
\end{equation}
Let $\lambda$ be the Artin-Rees exponent of $I$ relative to $\nn$. That is, we have $I\cap\nn^m=(I\cap\nn^\lambda)\nn^{m-\lambda}$ for all $m\geq\lambda$. For the remainder of the proof we are going to assume that $\mu_0\geq d+\lambda$. Then,
$I_\mu\subset I+\mm^{\mu+1}\subset I+\nn^\lambda\mm^d$, by~\eqref{eq:red-I},
and conversely, $I\subset I_\mu+\mm^{\mu+1}\subset I_\mu+\nn^\lambda\mm^d$, by~\eqref{eq:red-I-mu},
whence
\begin{equation}
\label{eq:AR}
I+\nn^\lambda\mm^d=I_\mu+\nn^\lambda\mm^d,\quad\mathrm{for\ any\ }\mu\geq\mu_0.
\end{equation}
We now claim that $(I\cap\nn^m)+\nn^m\mm^d\subset(I_\mu\cap\nn^m)+\nn^m\mm^d$, for all $m\geq1$.
Indeed, for $m<\lambda$, the inclusion $I\subset I_\mu+\nn^\lambda\mm^d$ implies
\[
I+\nn^m\mm^d\subset I_\mu+\nn^\lambda\mm^d+\nn^m\mm^d=I_\mu+\nn^m\mm^d\,,
\]
and hence
\[
(I\cap\nn^m)+\nn^m\mm^d=(I+\nn^m\mm^d)\cap\nn^m\subset(I_\mu+\nn^m\mm^d)\cap\nn^m=(I_\mu\cap\nn^m)+\nn^m\mm^d\,.
\]
If, in turn, $m\geq\lambda$, then~\eqref{eq:AR} yields
\begin{gather}
\notag
(I\cap\nn^m)+\nn^m\mm^d=(I\cap\nn^\lambda)\nn^{m-\lambda}+\nn^m\mm^d=((I\cap\nn^\lambda)+\nn^\lambda\mm^d)\nn^{m-\lambda}\\
\notag
=((I+\nn^\lambda\mm^d)\cap\nn^\lambda)\nn^{m-\lambda}=((I_\mu+\nn^\lambda\mm^d)\cap\nn^\lambda)\nn^{m-\lambda}=((I_\mu\cap\nn^\lambda)+\nn^\lambda\mm^d)\nn^{m-\lambda}\\
\notag
=(I_\mu\cap\nn^\lambda)\nn^{m-\lambda}+\nn^m\mm^d\subset(I_\mu\cap\nn^m)+\nn^m\mm^d\,.
\end{gather}
By \eqref{eq:iso1} and~\eqref{eq:iso2}, the above implies that there is, for every $m\geq1$, a well-defined epimorphism
\[
\frac{I+\nn^m}{I+\nn^m\mm^d}\cong\frac{\nn^m}{(I+\nn^m\mm^d)\cap\nn^m}\ \stackrel{\vp_m}{\longrightarrow}\ 
\frac{\nn^m}{(I_\mu\cap\nn^m)+\nn^m\mm^d}\cong\frac{I_\mu+\nn^m}{I_\mu+\nn^m\mm^d}\,.
\]
To complete the proof, it thus suffices to show that $\ker\vp_m=(0)$, or equivalently that $I_\mu\cap\nn^m\subset I+\nn^m\mm^d$, for all $m\geq1$.

Recall that, by assumption, we have $I_\mu\cap\nn=I_\mu\nn$. It follows by induction that $I_\mu\cap\nn^m=I_\mu\nn^m$, and hence $I_\mu\cap(I_\mu+\nn)^m=I_\mu(I_\mu+\nn)^{m-1}$, for all $m\geq1$. Moreover, by~\eqref{eq:red-I}, $I_\mu\subset I+\mm^{\mu+1}\subset I+\nn$,
 and by~\eqref{eq:red-I-mu}, $I\subset I_\mu+\mm^{\mu+1}\subset I_\mu+\nn$, hence $I+\nn=I_\mu+\nn$.
Finally, by~\eqref{eq:AR}, we also have $I_\mu\subset I+\nn\mm^d$, hence the sequence of inclusions
\begin{multline}
\notag
I_\mu\cap\nn^m\subset I_\mu\cap(I_\mu+\nn)^m=I_\mu(I_\mu+\nn)^{m-1}=I_\mu(I+\nn)^{m-1}\\
\subset I+I_\mu\nn^{m-1}\subset I+(I+\nn\mm^d)\nn^{m-1}\subset I+\nn^m\mm^d\,.
\end{multline}
\end{proof}
\medskip

%%%%%%%%%%%%%%%%%%%%%%%%%%%%%%%%%%%%%%%%%%%%%%%%%%%
%%%%%%%%%%%%%%%%%%%%% Section %%%%%%%%%%%%%%%%%%%%%
%%%%%%%%%%%%%%%%%%%%%%%%%%%%%%%%%%%%%%%%%%%%%%%%%%%
\section{Approximation of complete intersections}
\label{sec:CI}

Recall that an ideal $I$ in a regular local ring $A$ is called a \emph{complete intersection} when $I$ can be generated by $\dim{A}-\dim{A/I}$ elements. An analytic germ $(X,0)$ in $(\K^n,0)$ is a \emph{complete intersection singularity} when it is the zero set germ of a complete intersection ideal. The main result of this section, Theorem~\ref{thm:CI-approx} below, asserts that a complete intersection singularity can be arbitrarily closely approximated by germs of algebraic sets which are also complete intersections and share the same Hilbert-Samuel function.

We begin with a simple but useful observation.

\begin{proposition}
\label{prop:diagram-dim}
For an ideal $I$ in $\K\{x\}$, the following conditions are equivalent:
\begin{itemize}
\item[(i)] $\dim(\K\{x\}/I)\leq\dim\K\{x\}-k$.
\item[(ii)] After a generic linear change of coordinates in $\K^n$, the diagram $\NN(I)$ has a vertex on each of the first $k$ coordinate axes in $\N^n$.
\end{itemize}
\end{proposition}

\begin{proof}
Let as before $x_{[k]}=(x_1,\dots,x_k)$ and $\tx=(x_{k+1},\dots,x_n)$, and let $I(0)$ denote the ideal in $\K\{x_{[k]}\}$ obtained from $I$ by evaluating the $\tx$ variables at zero.

Condition (ii) then implies that the diagram $\NN(I(0))$ has finite complement in $\N^k$, and hence $\dim_\K\K\{x_{[k]}\}/I(0)<\infty$. By the Weierstrass Finiteness Theorem (see, e.g., \cite[Thm.\,1.10]{GLS}), it follows that $\dim(\K\{x\}/I)\leq\dim\K\{x\}-k$.

On the other hand, by Lemma~\ref{lem:tangent-cone}, condition (i) implies that after a generic linear change of coordinates in $\K^n$, for every $j=1,\dots,k$, $I$ contains a distinguished polynomial $P_j(x_j,\tx)=x_j^{d_j}+\sum_{r=1}^{d_j}a^j_r(\tx)x_j^{d_j-r}$ such that $P_j(x_j,\tx)\in\mm^{d_j}$. Since the total ordering of $\N^n$ is induced by the lexicographic ordering of the $(n+1)$-tuples $(|\beta|,\beta_n,\dots,\beta_1)$, it follows that $\inexp(P_j)=(0,\dots,d_j,0\dots,0)$ with $d_j$ in the $j$'th place. Hence (ii).
\end{proof}

Let $M$ be a finitely generated module over a local Noetherian ring $(A,\mm)$. Recall that a sequence $a_1,\dots,a_l\in\mm$ is called \emph{$M$-regular} if $a_1$ is not a zero-divisor in $M$ and $a_{i+1}$ is not a zero-divisor in $M/(a_1,\dots,a_i)M$ for $i=1,\dots,l-1$. $M$ is called \emph{Cohen-Macaulay} when $\mathrm{depth}_A(M)=\dim{M}$, where $\mathrm{depth}_A(M)$ is the maximum length of an $M$-regular sequence in $\mm$. A local ring $A$ is Cohen-Macaulay, when $A$ is Cohen-Macaulay as an $A$-module.
We will use the following well known fact from local algebra (see, e.g., \cite[Cor.\,B.8.12]{GLS}).

\begin{remark}
\label{rem:CM-flat}
Let $I$ be a proper ideal in $\K\{x\}$ with $\dim\K\{x\}/I=n-k$, and suppose that $\K\{x\}/I$ is a finite $\K\{\tx\}$-module, where $\tx=(x_{k+1},\dots,x_n)$. Then, $\K\{x\}/I$ is Cohen-Macaulay if and only if it is a flat $\K\{\tx\}$-module.
\end{remark}

\begin{theorem}
\label{thm:CI-approx}
Let $I=(F_1,\dots,F_k)$ be a complete intersection ideal in $\K\{x\}$ with $\dim\K\{x\}/I=n-k$. Then, there exists $\mu_0$ such that for every $\mu\geq\mu_0$ and for any $G_1,\dots,G_k\in\K\{x\}$ satisfying $j^\mu{G_i}=j^\mu{F_i}$, $1\leq i\leq k$, the ideal $I_\mu\coloneqq(G_1,\dots,G_k)$ is a complete intersection ideal in $\K\{x\}$ and $H_{I_\mu}=H_I$.
\end{theorem}

\begin{proof}
By Proposition~\ref{prop:diagram-dim}, after a generic linear change of coordinates in $\K^n$, the diagram $\NN(I)$ has a vertex $\beta^i$ on each of the first $k$ coordinate axes in $\N^n$. Let $H_1,\dots,H_k\in I$ be representatives of these vertices, so that $\inexp{H_i}=\beta^i$, $1\leq i\leq k$. Let $Q_{i,j}\in\K\{x\}$ be such that $H_i=\sum_{j=1}^kQ_{i,j}F_j$.
Set $\mu_1\coloneqq\max\{|\beta^1|,\dots,|\beta^k|\}$.

Since $\NN(I)$ has a vertex on each of the first $k$ coordinate axes in $\N^n$, the complement $\N^k\setminus\NN(I(0))$ is finite. Hence, by Corollary~\ref{cor:diagram-finite}, there exists $\mu_2\geq1$ such that, for every $\mu\geq\mu_2$ and $I_\mu\in U^\mu(I)$, $\NN(I(0))=\NN(I_\mu(0))$. Let then $\mu_0\coloneqq\max\{\mu_1,\mu_2\}$.

Fix $\mu\geq\mu_0$ and $G_1,\dots,G_k\in\K\{x\}$ satisfying $j^\mu{G_i}=j^\mu{F_i}$, $1\leq i\leq k$. Let $I_\mu=(G_1,\dots,G_k)$.
Then, for every $i$,
\[
j^\mu{H_i}=j^\mu(\sum_{j=1}^kQ_{i,j}F_j)=j^\mu(\sum_{j=1}^kQ_{i,j}j^\mu{F_j})=j^\mu(\sum_{j=1}^kQ_{i,j}j^\mu{G_j})=j^\mu(\sum_{j=1}^kQ_{i,j}G_j)\,,
\]
hence, by Remark~\ref{rem:jet-exp}, $\inexp(\sum_{j=1}^kQ_{i,j}G_j)=\beta^i$. It follows that $\beta^i\in\NN(I_\mu)$, $1\leq i\leq k$, and thus $\NN(I_\mu)$ has a vertex on each of the first $k$ coordinate axes in $\N^n$. By Proposition~\ref{prop:diagram-dim} again, we get $\dim\K\{x\}/I_\mu\leq n-k$. Since $I_\mu$ is generated by $k$ elements, it is thus a complete intersection ideal.

Since complete intersections are Cohen-Macaulay, then by Remark~\ref{rem:CM-flat}, both $\K\{x\}/I$ and $\K\{x\}/I_\mu$ are flat over $\K\{\tx\}$. Therefore, by Proposition~\ref{prop:flatness-diagram} and Remark~\ref{rem:base}, there exists $l\geq1$ such that for the linear form
\[
\LL(\beta)=\sum_{i=1}^k\beta_i+\sum_{j=k+1}^n\!l\beta_j\,,
\]
we have
\[
\NN_\LL(I)=\NN(I(0))\times\N^{n-k}\quad\mathrm{and}\quad\NN_\LL(I_\mu)=\NN(I_\mu(0))\times\N^{n-k}\,.
\]
Thus, $\NN_\LL(I)=\NN_\LL(I_\mu)$, and hence
\begin{equation}
\label{eq:dim1}
\dim_\K\frac{\K\{x\}}{I+\nn_{\LL,\eta+1}}\ =\ \dim_\K\frac{\K\{x\}}{I_\mu+\nn_{\LL,\eta+1}}\qquad\mathrm{for\ all\ }\eta\in\N\,,
\end{equation}
by Lemma~\ref{lem:HS-diagram-complement}.
Note that $\nn_{\LL,l}=(x_{[k]})^l+(\tx)$, and in general $\nn_{\LL,ml}=((x_{[k]})^l+(\tx))^m$, for all $m\in\N$. Also, since $\K\{x\}/I$ is a finite $\K\{\tx\}$-module, then for $l$ large enough one has $(x_{[k]})^l\subset I+(\tx)$ (Corollary~\ref{cor:reduction}). It follows that $I+(\tx)=I+\nn_{\LL,l}$, and hence by induction
\[
I+(\tx)^m=I+\nn_{\LL,ml}\qquad\mathrm{for\ all\ }m\in\N\,.
\]
Therefore, by \eqref{eq:dim1}, we get $\displaystyle{\dim_\K\frac{\K\{x\}}{I+(\tx)^m}=\dim_\K\frac{\K\{x\}}{I_\mu+(\tx)^m}}$ \ for all $m\in\N$.

Note finally that $I_\mu\cap(\tx)=I_\mu\!\cdot(\tx)$, by $\K\{\tx\}$-flatness of $\K\{x\}/I_\mu$ (see, e.g., \cite[Cor.\,7.6]{BM1}).
The theorem thus follows from Proposition~\ref{prop:HS-stability}.
\end{proof}
\medskip

%%%%%%%%%%%%%%%%%%%%%%%%%%%%%%%%%%%%%%%%%%%%%%%%%%%
%%%%%%%%%%%%%%%%%%%%% Section %%%%%%%%%%%%%%%%%%%%%
%%%%%%%%%%%%%%%%%%%%%%%%%%%%%%%%%%%%%%%%%%%%%%%%%%%
\section{Approximation of Cohen-Macaulay singularities}
\label{sec:CM}

The polynomial approximation of analytic germs is, in general, not possible beyond the complete intersection case (see Example~\ref{ex:no-CM-fin-det} below). The next best thing is approximation by Nash sets. The following result shows that a Cohen-Macaulay singularity can be arbitrarily closely approximated by germs of Nash sets which are also Cohen-Macaulay and share the same Hilbert-Samuel function.

\begin{theorem}
\label{thm:CM-approx}
Let $I=(F_1,\dots,F_s)$ be an ideal in $\K\{x\}$ such that $\K\{x\}/I$ is Cohen-Macaulay with $\dim\K\{x\}/I=n-k$. Then, there exist $\mu_0\in\N$, such that for any $\mu\geq\mu_0$ there are algebraic power series $G_1,\dots,G_s\in\K\left\langle{x}\right\rangle$ with $j^\mu{G_i}=j^\mu{F_i}$, $1\leq i\leq s$, the ideal $I_\mu=(G_1,\dots,G_s)$ satisfies $H_{I_\mu}=H_I$, and $\K\{x\}/I_\mu$ is Cohen-Macaulay with $\dim\K\{x\}/I_\mu=n-k$.
\end{theorem}

\begin{proof}
By Proposition~\ref{prop:diagram-dim}, after a generic linear change of coordinates in $\K^n$, the diagram $\NN(I)$ has a vertex on each of the first $k$ coordinate axes in $\N^n$. It follows that $\K\{x\}/I$ is $\K\{\tx\}$-finite, and hence $\K\{\tx\}$-flat (Remark~\ref{rem:CM-flat}). Therefore, by Proposition~\ref{prop:flatness-diagram} and Remark~\ref{rem:base}, there exists $l\geq1$ such that for the linear form
\[
\LL(\beta)=\sum_{i=1}^k\beta_i+\sum_{j=k+1}^n\!\!l\beta_j\,,
\]
we have
\[
\NN_\LL(I)=\NN(I(0))\times\N^{n-k}\,.
\]
We can extend the given set of generators $\{F_1,\dots,F_s\}$ by power series $F_{s+1},\dots,$ $F_r\in I$ such that the collection $\{F_1,\dots,F_r\}$ contains representatives of all the vertices $\NN_\LL(I)$. Since $I$ is generated by $\{F_1,\dots,F_s\}$, there are $H_p^q\in\K\{x\}$ such that
\[
F_{s+p}=\sum_{q=1}^sH_p^qF_q\,,\qquad p=1,\dots,r-s.
\]
Then, $\{F_1,\dots,F_r\}$ is a set of generators of $I$ and a standard basis of $I$ relative to $\LL$ (Corollary~\ref{cor:Hir-diagram}). For $i,j\in\{1,\dots,r\}$, $i<j$, let $S_{i,j}=S(F_i,F_j)$ denote the s-series of the pair $(F_i,F_j)$.
By Theorem~\ref{thm:Becker}, there exist $Q^{i,j}_m\in\K\{x\}$, $i,j,m\in\{1,\dots,r\}$, such that
\[
S_{i,j}=\sum_{m=1}^rQ^{i,j}_mF_m\qquad\mathrm{and}\qquad \inexp_\LL{S_{i,j}}\leq\min\{\inexp_\LL(Q^{i,j}_mF_m):m=1,\dots,t\}\,.
\]
Recall that, for all $1\leq i<j\leq r$, there are monomials $P_{i,j},P_{j,i}\in\K[x]$, which depend only on the initial terms of $F_i,F_j$, such that $S_{i,j}=P_{i,j}F_i-P_{j,i}F_j$.
Consider a system
\begin{equation}
\label{eq:st1basis}
\begin{cases}
\displaystyle{P_{i,j}(x)y_i-P_{j,i}(x)y_j-\sum_{m=1}^rz^{i,j}_my_m=0} &{~}\\
\displaystyle{y_{s+p}-\sum_{q=1}^sw_p^qy_q}=0 &{~}
\end{cases}
\end{equation}
of $\binom{r}{2}+r-s$ polynomial equations in variables $y=(y_1,\dots,y_r)$, $z=(z^{1,2}_1,\dots,z^{r-1,r}_r)$ and $w=(w_1^1,\dots,w_{r-s}^s)$. The system has a convergent solution $\{F_i,Q^{i,j}_m,H_p^q\}$, and hence by Theorem~\ref{thm:Artin}, for every $\mu\in\N$, an algebraic power series solution $\{G_i,R^{i,j}_m,K_p^q\}$ with $j^\mu{G_i}=j^\mu{F_i}$, $j^\mu{R^{i,j}_m}=j^\mu{Q^{i,j}_m}$, and $j^\mu{K_p^q}=j^\mu{H_p^q}$ for all $i,j,m,p,q$.

Let now $\mu_0\coloneqq\max\{\LL(\inexp_\LL{Q^{i,j}_m})+\LL(\inexp_\LL{F_m})):1\leq i<j\leq r,1\leq m\leq r\}$, and fix $\mu\geq\mu_0$.
Then, for any algebraic solution $\{G_i,R^{i,j}_m,K_p^q\}$ to \eqref{eq:st1basis} which coincides with $\{F_i,Q^{i,j}_m,H_p^q\}$ up to degree $\mu$, we have $S(G_i,G_j)=P_{i,j}G_i-P_{j,i}G_j$ and
\begin{multline}
\notag
S(G_i,G_j)=\sum_{m=1}^rR^{i,j}_mG_m,\\ \mathrm{with}\ \,\inexp_\LL{S(G_i,G_j)}\leq\min\{\inexp_\LL(R^{i,j}_mG_m):m=1,\dots,r\}\,.
\end{multline}
Hence the $G_i$ form a standard basis for the ideal $I_\mu=(G_1,\dots,G_r)$, by Theorem~\ref{thm:Becker} again.
In particular, the set $\{G_1,\dots,G_r\}$ contains representatives of all the vertices of $\NN_\LL(I_\mu)$ (see Remark~\ref{rem:st-bases}(1)). Since, by construction, $\inexp_\LL{G_i}=\inexp_\LL{F_i}$ for all $i$, it follows that $\NN_\LL(I_\mu)=\NN_\LL(I)$.
Thus, $\NN_\LL(I_\mu)=\NN(I(0))\times\N^{n-k}$ and so $\K\{x\}/I_\mu$ is $\K\{\tilde{x}\}$-flat, by Proposition~\ref{prop:flatness-diagram}.
Note also that $I_\mu$ is, in fact, generated by $\{G_1,\dots,G_s\}$, since the remaining generators $G_{s+1},\dots,G_r$ are combinations of the former, by~\eqref{eq:st1basis}.

The equality of diagrams $\NN_\LL(I_\mu)=\NN_\LL(I)$ implies, as in the proof of Theorem~\ref{thm:CI-approx}, that we have $\displaystyle{\dim_\K\frac{\K\{x\}}{I+(\tx)^m}=\dim_\K\frac{\K\{x\}}{I_\mu+(\tx)^m}}$\,, for all $m\in\N$. Moreover, $I_\mu\cap(\tx)=I_\mu\!\cdot(\tx)$, by $\K\{\tx\}$-flatness of $\K\{x\}/I_\mu$. The theorem thus follows from Proposition~\ref{prop:HS-stability}.
\end{proof}

In contrast with complete intersections, the Cohen-Macaulay singularities are not, in general, finitely determined. This can be shown using Becker's s-series criterion, as follows.

\begin{example}
\label{ex:no-CM-fin-det}
Let $I$ be an ideal in $\K\{x,y,z\}$ generated by
\[
F_1=x^8,\quad F_2=y^5+y^2z^4e^z,\quad F_3=x^2y^3+x^2z^4e^z\,.
\]
Let $\N^3$ be equipped with the standard ordering induced by lexicographic ordering of the $4$-tuples $(|\beta|,\beta_3,\beta_2,\beta_1)$.

We claim that $\{F_1,F_2,F_3\}$ are a standard basis of $I$.
Indeed, the s-series of pairs $(F_1,F_3)$ and $(F_2,F_3)$ are as follows:
\[
S_{1,3}=y^3F_1-x^6F_3=(-z^4e^z)F_1,\qquad S_{2,3}=x^2F_2-y^2F_3=0\,,
\]
which are standard representations in terms of $\{F_1,F_2,F_3\}$. The $S_{1,2}$, in turn, has a standard representation in terms of $F_1$ and $F_2$, because their initial exponents are relatively prime (see \cite[Thm.\,3.1]{Be2}). The claim thus follows from Theorem~\ref{thm:Becker}.

The diagram $\NN(I)$ contains vertices on the first two coordinate axes in $\N^3$, namely $\inexp{F_1}$ and $\inexp{F_2}$, hence $\K\{x,y,z\}/I$ is a finite $\K\{z\}$-module. On the other hand, by Remark~\ref{rem:st-bases}(1), the only vertices of $\NN(I)$ are the $\inexp{F_1}$, $\inexp{F_2}$, and $\inexp{F_3}$, which all lie in $\N^2\times\{0\}$. Thus, by Proposition~\ref{prop:flatness-diagram}, $\K\{x,y,z\}/I$ is $\K\{z\}$-flat, and hence Cohen-Macaulay (Remark~\ref{rem:CM-flat}).

Let now $\mu\geq8$ be arbitrary, and let
\[
G_1=x^8,\quad G_2=y^5+y^2z^4(e^z+z^{\mu-6}h(z)),\quad G_3=x^2y^3+x^2z^4e^z\,,
\]
where $h(z)\in\K\{z\}$ is an arbitrary non-zero series with $h(0)=0$. Then, $j^\mu{G_i}=j^\mu{F_i}$ for all $i$, but for the ideal $I_\mu=(G_1,G_2,G_3)$, the ring $\K\{x,y,z\}/I_\mu$ is not Cohen-Macaulay. Indeed, consider the s-series $S(G_2,G_3)$. We have
\[
S(G_2,G_3)=x^2G_2-y^2G_3=x^2y^2z^{\mu-2}h(z)\,,
\]
and hence $x^2y^2$ is a zero-divisor in $\K\{x,y,z\}/I_\mu$ regarded as a $\K\{z\}$-module. Thus, $\K\{x,y,z\}/I_\mu$ is not $\K\{z\}$-flat, and hence not a Cohen-Macaulay ring, by Remark~\ref{rem:CM-flat} again.
\end{example}
\medskip

%%%%%%%%%%%%%%%%%%%%%%%%%%%%%%%%%%%%%%%%%%%%%%%%%%%
%%%%%%%%%%%%%%%%%%%%% Section %%%%%%%%%%%%%%%%%%%%%
%%%%%%%%%%%%%%%%%%%%%%%%%%%%%%%%%%%%%%%%%%%%%%%%%%%
\section{Zariski equisingularity and Varchenko theorem}
\label{sec:Var}

In this section we recall a result of Varchenko on topological equisingularity of algebraically equisingular families. This is a central tool in the proof of Mostowski's theorem.

Let $V$ be a complex analytic hypersurface in a neighbourhood $U$ of the origin in $\C^l\times\C^n$, and let $T=V\cap(\C^l\times\{0\})$. Suppose there is, for every $0\leq i\leq n$, a distinguished polynomial
\[
F_i(t,x_{[i]})=x_i^{p_i}+\sum_{j=1}^{p_i}a_{i-1,j}(t,x_{[i-1]})x_i^{p_i-j}\,,
\]
where $t\in\C^l$, $x_{[i]}=(x_1,\dots,x_i)\in\C^i$, $a_{i-1,j}\in\C\{t,x_{[i-1]}\}$, all such that the following hold:
\begin{itemize}
\item[(1)] $V=F_n^{-1}(0)$.
\item[(2)] $a_{i,j}(t,0)\equiv0$, for all $i,j$.
\item[(3)] $F_{i-1}(t,x_{[i-1]})=0$ if and only if $F_i(t,x_{[i-1]},x_i)$, regarded as a polynomial in $x_i$ with $(t,x_{[i-1]})$ fixed, has fewer roots than for generic $(t,x_{[i-1]})$.
\item[(4)] Either $F_i(t,0)\equiv0$ or $F_i\equiv1$, and in the latter case $F_k\equiv1$ for all $k\leq i$ by convention.
\item[(5)] $F_0\equiv1$.
\end{itemize}
A system of distinguished polynomials $\{F_i(t,x_{[i]})\}$ satisfying the above conditions is called \emph{algebraically equisingular}.
Answering a question posed by Zariski \cite{Zar}, Varchenko showed that algebraic equisingularity of a system $\{F_i(t,x_{[i]})\}$ implies topological equisingularity of $V$ along $T$. More precisely, we have the following.

\begin{theorem}[{\cite{Var1,Var2}, cf. \cite[Thms.\,3.1,\,3.2]{BPR}}]
\label{thm:Var}
Under the above hypotheses, let $V_t=V\cap(\{t\}\times\C^n)$ and $U_t=U\cap(\{t\}\times\C^n)$, for $t\in T$. Then, for every $t\in T$, there exists a homeomorphism $h_t:U_0\to U_t$ such that $h_t(V_0)=V_t$ and $h_t(0)=0$. Moreover, if $F_n=G_1\dots G_r$ is a product of distinguished polynomials in $x_n$, then
\[
h_t(G_j^{-1}(0)\cap(\{0\}\times\C^n))=G_j^{-1}(0)\cap(\{t\}\times\C^n)\qquad\mathrm{for\ all\ }1\leq j\leq r\,.
\]
\end{theorem}

\medskip

%%%%%%%%%%%%%%%%%%%%%%%%%%%%%%%%%%%%%%%%%%%%%%%%%%%
%%%%%%%%%%%%%%%%%%%%% Section %%%%%%%%%%%%%%%%%%%%%
%%%%%%%%%%%%%%%%%%%%%%%%%%%%%%%%%%%%%%%%%%%%%%%%%%%
\section{Mostowski theorem with Hilbert-Samuel equisingularity}
\label{sec:Most}

The goal of this section is to prove a strong variant of Mostowski's theorem \cite{Most}, showing that every analytic germ $(X,0)\subset(\K^n,0)$ can be arbitrarily closely approximated by a Nash germ $(\hat{X},0)\subset(\K^n,0)$ with the same Hilbert-Samuel function, and such that the pairs $(\K^n,X)$ and $(\K^n,\hat{X})$ are topologically equivalent near zero.

\begin{theorem}
\label{thm:main}
Let $g_1,\dots,g_s\in\K\{x\}$ and let $(X,0)\subseteq (\K^n,0)$ be an analytic germ defined by $g_1=\dots=g_s=0$. Then, there exists $\mu_0$ such that for all $\mu\geq\mu_0$ there are algebraic power series $\hat{g}_1,\dots,\hat{g}_s\in\K\left\langle{x}\right\rangle$ and a homeomorphism germ $h:(\K^n,0)\to(\K^n,0)$ such that:
\begin{itemize}
\item[(i)] $j^\mu\hat{g}_k=j^\mu{g_k}$ for $k=1,\dots,s$
\item[(ii)] If $(\hat{X},0)$ is the Nash germ defined by $\hat{g}_1=\dots=\hat{g}_s=0$, then $H_{\hat{X},0}=H_{X,0}$
\item[(iii)] $h((X,0))=(\hat{X},0)$.
\end{itemize}
\end{theorem}

Our proof of Theorem~\ref{thm:main} combines the exposition of \cite{BPR} with the Becker s-series criterion (Theorem~\ref{thm:Becker}). We include the details of the argument for the reader's convenience.

\subsection{Generalized discriminants}

Let $T=(T_1,\dots,T_p)$ be variables. For $j\geq1$, consider
\[
\Delta_j=\sum_{r_1, \ldots, r_{j-1}} \prod_{k\neq l \atop k,l \neq r_1, \ldots, r_{j-1}} (T_k-T_l).
\]
The $\Delta_j$ are symmetric in variables $T$, and hence each $\Delta_j=\Delta_j(A_0,\dots,A_{p-1})$ is a polynomial in the elementary symmetric functions $A_0=T_1\cdots T_p,\dots,A_{p-1}=T_1+\dots+T_p$. We have: A polynomial $X^p+a_{p-1}X^{p-1}+\dots+a_1X+a_0$ has precisely $p-j$ distinct roots if and only if $\Delta_1(a_0,\dots,a_{p-1})=\dots=\Delta_j(a_0,\dots,a_{p-1})=0$ and $\Delta_{j+1}(a_0,\dots,a_{p-1})\neq0$.

\subsection{Construction of a normal system of equations}
\label{subsec:normal}

Let $g_1,\dots,g_s\in\K\{x\}$ and let $I\coloneqq(g_1,\dots,g_s)\!\cdot\!\K\{x\}$. After a generic linear change of coordinates, if needed, all the $g_k$ become regular in variable $x_n$. We may thus, without loss of generality, assume that each $g_k$ is a distinguished polynomial in $x_n$. That is,
\begin{equation}
\label{eq:g-k}
g_k(x)=x_n^{r_k}+ \sum_{j=1}^{r_k} a_{n-1,k,j} (x_{[n-1]}) x_n^{r_k-j}\,,
\end{equation}
where $a_{n-1,k,j}\in\K\{x_{[n-1]}\}$ and $a_{n-1,k,j}(0)=0$.

The coefficients $a_{n-1,k,j}$ can be arranged in a row vector $a_{n-1}\in\K\{x_{[n-1]}\}^{p_n}$, where $p_n=\sum_k r_k$.
Set $f_n\coloneqq g_1\cdots g_s$. Then, the generalized discriminants $\Delta_{n,i}$ of $f_n$ are polynomials in $a_{n-1}$.
Let $j_n$ be such that
\[
\Delta_{n,i} ( a_{n-1} )\equiv 0 \quad\mathrm{for}\ i<j_n\,,
\]
and $\Delta_{n,j_n}(a_{n-1})\not\equiv 0$. Then, after a linear change of coordinates $x_{[n-1]}$, we may write
\[
\Delta_{n,j_n}(a_{n-1})=u_{n-1}(x_{[n-1]})(x_{n-1}^{p_{n-1}}+\sum_{j=1}^{p_{n-1}} a_{n-2,j} (x_{[n-2]}) x_{n-1}^{p_{n-1}-j})\,,
\]
where $u_{n-1}(0)\ne 0$, and for all $j$, $a_{n-2,j}(0)=0$. Set
\[
f_{n-1}\coloneqq x_{n-1}^{p_{n-1}}+\sum_{j=1}^{p_{n-1}} a_{n-2,j} (x_{[n-2]}) x_{n-1}^{p_{n-1}-j},
\]
and denote the vector of its coefficients $a_{n-2,j}$ by $a_{n-2}\in\K\{x_{[n-2]}\}^{ p_{n-1}}$.
Let $j_{n-1}$ be such that the first $j_{n-1}-1$ generalized discriminants $\Delta_{n-1,i}$ of $f_{n-1}$ are identically zero and $\Delta_{n-1,j_{n-1}}$ is not. Then, again, we define $f_{n-2}(x_{[n-2]})$ as the distinguished polynomial associated to $\Delta_{n-1,j_{n-1}}$, and so on.

By induction, we define a system of distinguished polynomials $f_i\in\K\{x_{[i-1]}\}[x_i]$, $i=1,\dots,n-1$, such that
\[
f_i= x_i^{p_i}+ \sum_{j=1}^{p_i} a_{i-1,j}(x_{[i-1]})x_i^{p_i-j}
\]
is the distinguished polynomial associated to the first non identically zero generalized discriminant $\Delta_{i+1,j_{i+1}}(a_{i})$ of $f_{i+1}$:
\begin{equation}
\label{eq:polynomial-f-i}
\Delta_{i+1,j_{i+1}}(a_i)=u_i(x_{[i]})(x_i^{p_i}+\sum_{j=1}^{p_i}a_{i-1,j}(x_{[i-1]})x_i^{p_i-j})\,, \quad i=0,...,n-1,
\end{equation}
where, in general, $a_{i}=(a_{i,1},\dots,a_{i,p_{i+1}})$.
Thus the vector of functions $a_i$ satisfies
\begin{equation}
\label{eq:discriminantsi}
\Delta_{i+1,k}(a_i)\equiv 0 \qquad\mathrm{for\ \ } k<j_{i+1}, \quad i=0,...,n-1. 
\end{equation}
This means in particular that
\[
\Delta_{1,k}(a_0) \equiv 0 \quad \mathrm{for\ } k<j_1\quad\mathrm{\ and\ \ } \Delta_{1,j_1}(a_0) \equiv u_0,
\]
where $u_0$ is a non-zero constant.

\subsection{Incorporating a standard basis}
\label{subsec:stb}

Consider now the diagram of initial exponents $\NN(I)$ of the ideal $I$ in $\K\{x\}$ (with respect to the linear form $\LL(\beta)=|\beta|$ on $\N^n$).
We can extend the given set of generators $\{g_1,\dots,g_s\}$ by power series $g_{s+1},\dots,g_r\in I$ such that the collection $\{g_1,\dots,g_r\}$ contains representatives of all the vertices of $\NN(I)$. Since $I$ is generated by $\{g_1,\dots,g_s\}$, there are $h_p^q\in\K\{x\}$ such that
\begin{equation}
\label{eq:g-s+p}
g_{s+p}(x)=\sum_{q=1}^sh_p^q(x)\cdot\left(x_n^{r_q}+ \sum_{j=1}^{r_q} a_{n-1,q,j} (x_{[n-1]}) x_n^{r_q-j}\right)\,,
\end{equation}
for $p=1,\dots,r-s$, by~\eqref{eq:g-k}.

Now, $\{g_1,\dots,g_r\}$ is a set of generators of $I$ and a standard basis of $I$ (Corollary~\ref{cor:Hir-diagram}). For $i,j\in\{1,\dots,r\}$, $i<j$, let $S_{i,j}=S(g_i,g_j)$ denote the s-series of the pair $(g_i,g_j)$.
By Theorem~\ref{thm:Becker}, there exist $v^{i,j}_m\in\K\{x\}$, $i,j,m\in\{1,\dots,r\}$, such that
\begin{equation}
\label{eq:S-i-j}
S_{i,j}=\sum_{m=1}^rv^{i,j}_mg_m\qquad\mathrm{and}\qquad \inexp{S_{i,j}}\leq\min\{\inexp(v^{i,j}_mg_m):m=1,\dots,t\}\,.
\end{equation}
Recall that, for all $1\leq i<j\leq r$, there are monomials $P_{i,j},P_{j,i}\in\K[x]$, which depend only on the initial terms of $g_i,g_j$, such that $S_{i,j}=P_{i,j}g_i-P_{j,i}g_j$. Therefore, the $v^{i,j}_m$, $h_p^q$, and $a_{n-1,q,j}$ satisfy the following system of $\binom{r}{2}$ polynomial equations
\begin{equation}
\label{eq:all-1}
P_{i,j}(x)g_i-P_{j,i}(x)g_j-\sum_{m=1}^rv^{i,j}_mg_m=0\,,\qquad 1\leq i<j\leq r\,,
\end{equation}
in which the $h_p^q$ and $a_{n-1,q,j}$ are present via~\eqref{eq:g-k} and~\eqref{eq:g-s+p}. We will denote the vector of functions $v^{i,j}_m$ by $v\in\K\{x\}^{r\binom{r}{2}}$, and the vector of $h_p^q$ by $h\in\K\{x\}^{s(r-s)}$, to simplify notation.

\subsection{Approximation by Nash functions}
\label{subsec:aNf}

Consider~\eqref{eq:polynomial-f-i}, \eqref{eq:discriminantsi}, and~\eqref{eq:all-1} as a system of polynomial equations in $a_i(x_{[i]})$, $u_i(x_{[i]})$, $v(x)$, and $h(x)$. By construction, this system admits a convergent solution. Therefore, by Theorem~\ref{thm:BPR}, there exist a new set of variables $z=(z_1,\dots,z_k)$, an increasing function $\tau:\N\to\N$, convergent power series $z_i(x)\in\K\{x\}$ vanishing at zero, algebraic power series $\hat{u}_i(x_{[i]},z)\in\K\!\left\langle{x_{[i]},z_1,\dots,z_{\tau(i)}}\right\rangle$, and vectors of algebraic power series $\hat{a}_i(x_{[i]},z)\in\K\!\left\langle{x_{[i]},z_1,\dots,z_{\tau(i)}}\right\rangle^{p_i}$, $\hat{v}(x,z)\in\K\!\left\langle{x,z}\right\rangle^{r\binom{r}{2}}$, and $\hat{h}(x,z)\in\K\!\left\langle{x,z}\right\rangle^{s(r-s)}$ all such that the following hold:
\begin{itemize}
\item[(a)] $z_1(x),\dots,z_{\tau(i)}(x)$ depend only on variables $x_{[i]}=(x_1,\dots,x_i)$
\item[(b)] $\hat{u}_i(x_{[i]},z), \hat{a}_i(x_{[i]},z), \hat{v}(x,z), \hat{h}(x,z)$ are solutions of~\eqref{eq:polynomial-f-i}, \eqref{eq:discriminantsi}, and~\eqref{eq:all-1}
\item[(c)] The convergent solutions satisfy:\\ $u_i(x_{[i]})=\hat{u}_i(x_{[i]},z(x_{[i]}))$, $a_i(x_{[i]})=\hat{a}_i(x_{[i]},z(x_{[i]}))$, $v(x)=\hat{v}(x,z(x))$, and $v(h)=\hat{h}(x,z(x))$.
\end{itemize}

\subsection{Proof of Theorem~\ref{thm:main}}
\label{subsec:pf-main}

Let $g_1,\dots,g_s\in\K\{x\}$ and let $(X,0)\subseteq (\K^n,0)$ be an analytic germ defined by $g_1=\dots=g_s=0$.
Suppose first that $\K=\C$.

Let $g_{s+1}(x),\dots,g_r(x)$, $u_i(x_{[i]})$, $a_i(x_{[i]})$, $v(x)$, and $h(x)$ be as in sections~\ref{subsec:normal} and~\ref{subsec:stb}. Set
\[
\mu_0\coloneqq\max\{|\inexp(v^{i,j}_m)|+|\inexp(g_m)|:1\leq i<j\leq r,1\leq m\leq r\}\,,
\]
and fix $\mu\geq\mu_0$.

Let $z_i(x)\in\C\{x\}$ be the convergent power series, and let $\hat{u}_i(x_{[i]},z)$, $\hat{a}_i(x_{[i]},z)$, $\hat{v}(x,z)$, and $\hat{h}(x,z)$ be the (vectors of) algebraic power series constructed above. To simplify notation, we will write $\bar{z}^\mu_i(x)$ for $z_i(x)-j^\mu z_i(x)$, where as before $j^\mu z_i(x)$ denotes the $\mu$-jet of $z_i$ as a power series in variables $x$.

For $t\in\C$, we define
\[
F_n(t,x)\coloneqq\prod_{k=1}^s G_k(t,x)\,,
\]
where
\begin{equation}
\label{eq:bigg}
G_k(t,x) \coloneqq x_n^{r_k}+\sum_{j=1}^{r_k}\hat{a}_{n-1,k,j}(x_{[n-1]},j^\mu z(x_{[n-1]})+t\bar{z}^\mu(x_{[n-1]}))\!\cdot\!x_n^{r_k-j}\,,
\end{equation}
and
\[
F_i(t,x) \coloneqq  x_i^{p_i}+\sum_{j=1}^{p_i}\hat{a}_{i-1,j}(x_{[i-1]}, j^\mu z(x_{[i-1]})+t\bar{z}^\mu(x_{[i-1]}))\!\cdot\!x_i^{p_i-j}\,,\quad i=1,\dots,n-1.
\]
Finally, we set $F_0(t)\equiv1$.
Now, since $u_i(0,0)=\hat{u}_i(0,z(0))\neq0$, $i=1,\dots,n-1$, it follows that the family $\{F_i(t,x_{[i]})\}$ is algebraically equisingular (with $|t|<R$, for any $R<\infty$).

Set $\hat{g}_k(x)\coloneqq G_k(0,x)$, and let $(\hat{X},0)$ be the Nash germ in $(\C^n,0)$ defined by $\hat{g}_1=\dots=\hat{g}_s=0$. Note that $g_k=G_k(1,x)$, $k=1,\dots,s$. Therefore, by Theorem~\ref{thm:Var}, there is a homeomorphism germ $h:(\C^n,0)\to(\C^n,0)$ such that $h((X,0))=(\hat{X},0)$.

By construction, we have $j^\mu\hat{g}_k=j^\mu{g_k}$ for $k=1,\dots,s$. Finally, as in the proof of Theorem~\ref{thm:CM-approx}, observe that the collection $\{\hat{g}_1,\dots,\hat{g}_s,\dots,\hat{g}_r\}$ forms a standard basis for the ideal $I_\mu$ that it generates (by~\eqref{eq:S-i-j}).
In particular, the set $\{\hat{g}_1,\dots,\hat{g}_r\}$ contains representatives of all the vertices of the diagram $\NN(I_\mu)$ (see Remark~\ref{rem:st-bases}(1)). Since, by construction and the choice of $\mu_0$, we have $\inexp(\hat{g}_k)=\inexp(g_k)$ for all $k$, it follows that $\NN(I_\mu)=\NN(I)$. Hence, $H_{\hat{X},0}=H_{X,0}$, by Lemma~\ref{lem:HS-diagram-complement}.
Note also that $I_\mu$ is, in fact, generated by $\{\hat{g}_1,\dots,\hat{g}_s\}$, since the remaining generators $\hat{g}_{s+1},\dots,\hat{g}_r$ are combinations of the former, by~\eqref{eq:g-s+p}.
This completes the proof in the complex case.

The real case follows from the complex one, since by \cite[\S6]{Var1}, if the distinguished polynomials $F_i$ of Section~\ref{sec:Var} have real coefficients, then the homeomorphism $h$ constructed in Varchenko's Theorem~\ref{thm:Var} is conjugation invariant.
\qed
\medskip

%%%%%%%%%%%%%%%%%%%%%%%%%%%%%%%%%%%%%%%%%%%%%%%%%%%
% References
%%%%%%%%%%%%%%%%%%%%%%%%%%%%%%%%%%%%%%%%%%%%%%%%%%%
\bibliographystyle{amsplain}

\end{document}